\documentclass[12pt]{amsart}
\usepackage{amsmath, amsthm, amssymb}
\usepackage[headings]{fullpage}
\usepackage[pdfstartview={FitH}]{hyperref}
\usepackage{enumitem} 
\usepackage{xypic}

\usepackage{color}   

\numberwithin{equation}{section}

\theoremstyle{plain}
\newtheorem{theorem}[equation]{Theorem}
\newtheorem{corollary}[equation]{Corollary}
\newtheorem{lemma}[equation]{Lemma}
\newtheorem{proposition}[equation]{Proposition}
\newtheorem*{ClassicalTheorem}{Classical Diagonalization Theorem}
\newtheorem*{InfiniteDiagonal}{Infinite-Dimensional Diagonalization Theorem}
\newtheorem*{InfiniteWA}{Infinite-Dimensional Wedderburn-Artin Theorem}

\theoremstyle{definition}
\newtheorem{definition}[equation]{Definition}

\newtheorem{example}[equation]{Example}
 
\newtheorem{remark}[equation]{Remark}

\newcommand{\cat}[1]{\ensuremath{\mathsf{#1}}}
\newcommand{\op}{\ensuremath{^\mathrm{\,op}}}

\newcommand{\TMod}{\cat{TMod}} 
\newcommand{\PC}{\cat{PC}}
\newcommand{\Func}{\cat{Func}}
\newcommand{\Set}{\cat{Set}}
\newcommand{\TRing}{\cat{TRing}}
\newcommand{\cTRing}{\cat{cTRing}}
\newcommand{\TAlg}{\cat{TAlg}}
\newcommand{\lann}{\operatorname{ann}_\ell}

\newcommand{\eps}{\varepsilon}
\newcommand{\ev}{\mathrm{ev}}

\newcommand{\leql}{\leq_\ell}
\newcommand{\leqr}{\leq_r}

\DeclareMathOperator{\Hom}{Hom}
\DeclareMathOperator{\End}{End}

\DeclareMathOperator{\diag}{diag}

\DeclareMathOperator{\range}{range}
\DeclareMathOperator{\Span}{Span}
\DeclareMathOperator{\soc}{soc}
\DeclareMathOperator{\Spec}{Spec}
 
\newcommand{\Z}{\mathbb{Z}}

\newcommand{\M}{\mathbb{M}}
\newcommand{\K}{K}
\newcommand{\F}{\mathbb{F}}

\newcommand{\Abar}{\overline{A}}

\newcommand{\B}{\mathcal{B}}
\newcommand{\U}{\mathcal{U}}
\newcommand{\p}{\mathfrak{p}}
\newcommand{\m}{\mathfrak{m}}
\newcommand{\D}{\mathcal{D}}

\newcommand{\two}{\mathbf{2}}

\newcommand{\Cc}{\mathcal{C}}

\newcommand{\Ss}{\mathcal{S}}

\begin{document}

\title{Infinite-dimensional diagonalization and semisimplicity}

\date{July 30, 2016}

\keywords{Diagonalizable operators, Wedderburn-Artin theorem, finite topology, 
pro-discrete module, pseudocompact algebra, function algebra}

\subjclass[2010]{
Primary:
15A04, 
16S50; 
Secondary:
15A27, 
16W80, 
18E15 
}

\begin{abstract}
We characterize the diagonalizable subalgebras of $\End(V)$, the full ring of 
linear operators on a vector space $V$ over a field, in a manner that directly generalizes
the classical theory of diagonalizable algebras of operators on a finite-dimensional 
vector space. Our characterizations are formulated in terms of a natural topology 
(the ``finite topology'') on $\End(V)$, which reduces to the discrete topology in the case 
where $V$ is finite-dimensional.
We further investigate when two subalgebras of operators can and cannot be 
simultaneously diagonalized, as well as the closure of the set of diagonalizable 
operators within $\End(V)$. 
Motivated by the classical link between diagonalizability and semisimplicity, we also
give an infinite-dimensional generalization of the Wedderburn-Artin theorem, providing 
a number of equivalent characterizations of left pseudocompact, Jacoboson semisimple 
rings that parallel various characterizations of artinian semisimple rings. 
This theorem unifies a number of related results in the literature, including the structure 
of linearly compact, Jacobson semsimple rings and cosemisimple coalgebras over a field.
%
\end{abstract}

\author{Miodrag C. Iovanov}
\address{University of Iowa \\
Department of Mathematics \\
McLean Hall \\
Iowa City, IA, 52242, USA}
\address{University of Bucharest\\
 Str.\ Academiei 14\\
 Bucharest, Romania}
\email{yovanov@gmail.com; miodrag-iovanov@uiowa.edu }
\thanks{Iovanov was partially supported by the UEFISCDI [grant number PN-II-ID-PCE-2011-3-0635], contract no. 253/5.10.2011 of CNCSIS.}

\author{Zachary Mesyan}
\address{Department of Mathematics \\
University of Colorado \\
Colorado Springs, CO, 80918, USA}
\email{zmesyan@uccs.edu}

\author{Manuel L. Reyes}
\address{Department of Mathematics \\
Bowdoin College \\
Brunswick, ME 04011--8486, USA}
\email{reyes@bowdoin.edu}
\thanks{Reyes was supported by NSF grant DMS-1407152}

\maketitle

\section{Introduction}
\label{sec:intro}

Let $V$ be a vector space over a field $\K$, and let $\End(V)$ be the algebra of
$\K$-linear operators on $V$. Given any basis $\B$ of $V$, an operator $T \in \End(V)$ is 
said to be \emph{diagonalizable with respect to $\B$} (or \emph{$\B$-diagonalizable}) if every 
element of $\B$ is an eigenvector for $T$. It is easy to see that the set $\diag(\B)$ of all 
$\B$-diagonalizable  operators forms a maximal commutative subalgebra of $\End(V)$, isomorphic 
to the product algebra  $\K^\B = \prod_\B \K \cong \K^{\dim(V)}$ (for details, see
Proposition~\ref{prop:diagonal structure} below). 

A subalgebra $D \subseteq \End(V)$ is said to be \emph{diagonalizable} if 
$D \subseteq \diag(\B)$ for some basis $\B$ of $V$; this is evidently equivalent to
the condition that, for any basis $\B$ of $V$, there exists a unit $u \in \End(V)$ such that
$uDu^{-1} \subseteq \diag(\B)$.
Clearly every diagonalizable subalgebra of $\End(V)$ is commutative. Our goal is to characterize 
which commutative subalgebras of $\End(V)$ are diagonalizable.

For motivation, we recall the classical result from linear algebra characterizing the diagonalizable 
operators on a finite-dimensional vector space. An individual operator is diagonalizable if 
and only if its minimal polynomial splits into linear factors over $\K$ and has no repeated 
roots; in the case where $\K$ is algebraically closed, this is of course equivalent to the property 
that the minimal polynomial has no repeated roots. For subalgebras, this implies the following.

\begin{ClassicalTheorem}
Let $V$ be a finite-dimensional vector space over a field $\K$, and let $C$ be a
commutative subalgebra of $\End(V)$. The following are equivalent:
\begin{enumerate}[label=(\alph*)]
\item $C$ is diagonalizable;
\item $C \cong \K^n$ as $\K$-algebras for some integer $1 \leq n \leq \dim(V)$;
\item $C$ is spanned by an orthogonal set of idempotents whose sum is $1$.  
\end{enumerate}
If $\K$ is algebraically closed, then these conditions are further equivalent to:
\begin{enumerate}[label=(\alph*), resume]
\item $C$ is a Jacobson semisimple $\K$-algebra;
\item $C$ is a reduced $\K$-algebra (i.e., has no nonzero nilpotent elements).
\end{enumerate}
\end{ClassicalTheorem}

(The equivalence (a)$\Leftrightarrow$(b) is mentioned, for instance, 
in~\cite[VII.5.7]{Bourbaki:algebre}, the equivalence (b)$\Leftrightarrow$(c) is apparent, and 
the further equivalence of~(b) with (d)--(e) is a straightforward restriction of the Wedderburn-Artin
structure theorem to finite-dimensional algebras over algebraically closed fields.)

In particular, the diagonalizability of a commutative subalgebra $C$ of $\End(V)$ is 
completely determined by the isomorphism class of $C$ as a $\K$-algebra. 
However, in the case where $V$ is infinite-dimensional, it is impossible to determine whether
a commutative subalgebra $C \subseteq \End(V)$ is diagonalizable from purely
algebraic properties of $C$. Indeed, assume for simplicity that $\K$ is an infinite
field, and let $V$ be a vector space with basis $\{v_n \mid n = 1, 2, 3, \dots\}$.
Let $T \in \End(V)$ be diagonalizable with infinitely many eigenvalues, and let
$S \in \End(V)$ be the ``right shift operator'' $S(v_n) = v_{n+1}$. Then
$\K[T] \cong \K[S]$ are both isomorphic to a polynomial algebra $\K[x]$ (as neither
$T$ nor $S$ satisfies any polynomial relation), while $\K[T]$ is diagonalizable 
but $\K[S]$ is not (for instance, because $S$ has no eigenvalues).

Therefore the study of diagonalizability over infinite-dimensional spaces requires
us to consider some extra structure on algebras of operators. Our approach will be to 
consider a natural topology on $\End(V)$, which is then inherited by its subalgebras. 
(Of course, this is analogous to the use of functional analysis to extend results about 
operators on Hilbert spaces from finite-dimensional linear algebra to the infinite-dimensional 
setting.)
Specifically, we consider the \emph{finite topology} on the algebra $\End(V)$, also called the
\emph{function topology} or the topology of \emph{pointwise convergence}.
This is the topology with a basis of open sets of the form $\{S \in \End(V) \mid S(x_i) = 
T(x_i),\, i=1,\dots,n\}$ for fixed $x_1,\dots,x_n\in V$ and $T \in \End(V)$.

It will be shown below (Lemma~\ref{lem:commutativeclosure}) that the closure of a commutative 
subalgebra of $\End(V)$ is again commutative; in particular, every 
maximal commutative subalgebra is closed. Thus after replacing a commutative subalgebra with 
its closure, we may reduce the problem of characterizing diagonalizable subalgebras of $\End(V)$ 
to that of characterizing which \emph{closed} subalgebras are diagonalizable. The following major
result, given in Theorem~\ref{thm:diag}, is a generalization of the Classical Diagonalization 
Theorem, which shows that diagonalizability of a closed subalgebra of $\End(V)$ can be detected 
from its internal structure as a topological algebra. 
We say that a commutative topological $\K$-algebra is \emph{$\K$-pseudocompact} if it is an inverse
limit (in the category of topological algebras) of finite-dimensional discrete $\K$-algebras;
see Section~\ref{sec:WA} for further details.
 
\begin{InfiniteDiagonal}
Let $V$ be a vector space over a field $\K$, and let $C$ be a closed commutative subalgebra of 
$\End(V)$. The following are equivalent:
\begin{enumerate}[label=(\alph*)]
\item $C$ is diagonalizable;
\item $C \cong \K^\Omega$ as topological $\K$-algebras, for some cardinal $1 \leq \Omega 
\leq \dim(V)$;
\item $C$ contains an orthogonal set of idempotents $\{E_i\}$ whose linear span is dense in $C$,
such that the net of finite sums of the $E_i$ converges to $1$.
\end{enumerate}
If $\K$ is algebraically closed, then these conditions are further equivalent to:
\begin{enumerate}[label=(\alph*), resume]
\item $C$ is $\K$-pseudocompact and Jacobson semisimple;
\item $C$ is $\K$-pseudocompact and topologically reduced (i.e., has no nonzero topologically 
nilpotent elements).
\end{enumerate}
\end{InfiniteDiagonal}

In the classical case of a finite-dimensional vector space $V$ over a field $K$, because every 
diagonalizable algebra of operators is isomorphic to $K^n$, we see that diagonalizable algebras
are semisimple.
Similarly, in the case where $V$ is infinite-dimensional, every closed diagonalizable algebra of
operators on $V$ satisfies a suitable version of semisimplicity. Thus in Theorem~\ref{thm:WA}, our
other major result, we give a version of the  Wedderburn-Artin theorem that is suitable for our context 
of infinite-dimensional linear algebra, with the dual purpose of extending the classical connection 
above and preparing for the proof of Theorem~\ref{thm:diag}. 
This result characterizes those topological rings that are isomorphic to $\prod_{i}\End_{D_i}(V_i)$ 
where the $V_i$ are arbitrary right vector spaces over division rings $D_i$.
It is presented in the context of \emph{left pseudocompact (topological) rings} (as defined
by Gabriel~\cite{G}), and the section begins with a basic account of such rings and
and modules. In Theorem~\ref{thm:WA}, among other equivalences, we show the following.

\begin{InfiniteWA}
Let $R$ be a topological ring. Then the following are equivalent:
\begin{enumerate}[label=(\alph*)]
\item As a left topological module, $R$ is a product of simple discrete left modules;
\item $R \cong \prod \End_{D_i}(V_i)$ as topological rings, where each $V_i$ is a right vector
space over a division ring $D_i$ and each $\End_{D_i}(V_i)$ is given the finite topology;
\item $R$ is left pesudocompact and Jacobson semisimple;
\item $R$ is left pseudocompact and every pseudocompact left $R$-module is a product
of discrete simple modules.
\end{enumerate}
\end{InfiniteWA}

This result generalizes the classical Wedderburn-Artin Theorem (which is recovered in
the case when $R$ carries the discrete topology), and as we explain after Theorem~\ref{thm:WA}, 
it provides a unifying context for a number of other related results in the literature, including the 
structure  of linearly compact, Jacobson semisimple rings, and the characterization of cosemisimple 
coalgebras and their dual algebras.

We now give a tour of the major ideas covered in this paper, by way of outlining each section.
As the problem of characterizing diagonalizable subalgebras of $\End(V)$ is framed in a purely
algebraic context, we anticipate that some readers may not have much familiarity with the 
topological algebra required to prove our results.
Thus we feel that it is appropriate to briefly recall some of the basic theory of topological
rings and modules. Most of this review is carried out in Section~\ref{sec:topology}.
Selected topics include topological sums of orthogonal collections of idempotents,
as well as topological modules that are inverse limits of discrete modules.

Section~\ref{sec:WA} begins with a basic overview of pseudocompact and linearly compact
modules and rings, including a brief review of Gabriel duality for pseudocompact modules. After
a discussion of a suitable topological version of the Jacobson radical, we proceed to our infinite
Wedderburn-Artin theory which culminates in Theorem~\ref{thm:WA}. The remainder of the
section consists of various applications of this theorem. This includes a symmetric version
of the Infinite-Dimensional Wedderburn-Artin Theorem, which characterizes topological rings 
of the form $\prod \M_{n_i}(D_i)$, where the $D_i$ are division rings, the $n_i \geq 1$ are 
integers, and each component matrix ring is equipped with the discrete topology. The section 
ends with applications to the structure of cosemisimple coalgebras.

The structure and characterization of diagonalizable algebras of operators is addressed in
Section~\ref{sec:diagonal}. If $\B$ is a basis of a vector space $V$, it is shown that the
subalgebra of $\End(V)$ of $\B$-diagonalizable operators is topologically isomorphic to $\K^\B$,
where $\K$ is discretely topologized. Through a detailed study of topological algebras of
the form $\K^X$ for some set $X$, we show that such ``function algebras'' with continuous
algebra homomorphisms form a category that is dual to the category of sets. This is applied 
in Theorem~\ref{thm:diag} to characterize closed diagonalizable subalgebras of $\End(V)$. 
We also study diagonalizable subalgebras that are not necessarily closed and individual
diagonalizable operators, still making use of the finite topology on $\End(V)$. 
We then turn our attention to the problem of simultaneous diagonalization of
operators, showing in that if $C,D \subseteq \End(V)$ are diagonalizable subalgebras that 
centralize one another, then they are simultaneously diagonalizable.
We also present examples of sets of commuting diagonalizable operators that cannot 
be simultaneously diagonalized. Finally, we describe the closure $\overline{\D}$ of the set 
$\D = \D(V)$ of diagonalizable operators in $\End(V)$. Here the order of the field of scalars 
plays an important role.
If $\K$ is finite, we have $\D = \overline{\D}$; in the case where $\K$ is infinite, $\overline{\D}$ 
is shown to consist of those operators $T$ that are diagonalizable on the torsion part of $V$
when regarded as a $\K[x]$-module via the action of $T$.

\subsection{A reader's guide to diagonalization}

While our major motivation for the present investigation was to characterize the diagonalizable
subalgebras of $\End(V)$, we have attempted to place this result in a broader context within
infinite-dimensional linear algebra.  In the classical finite-dimensional case, diagonalization is
connected to the structure theory of finite-dimensional algebras and to algebras of functions (though
one typically does not think beyond polynomial functions in classical spectral theory). 
Our attempt to understand infinite-dimensional diagonalization from both of these perspectives 
led us to incorporate aspects of the theories of pseudocompact algebras and ``function algebras.''

We anticipate that some readers may appreciate a shorter path to understanding only the 
proof of the characterization of diagonalizability in Theorem~\ref{thm:diag}; for such readers, 
we provide the following outline.
The definition of the finite topology on $\End(V)$ is the most important portion of
Section~\ref{sec:topology}, and the discussion around Lemmas~\ref{lem:open} 
and~\ref{lem:prodiscrete} is intended to help the reader better understand the structure of 
pseudocompact modules. Other results from Section~\ref{sec:topology} can mostly be referred to 
as needed, although readers unfamiliar with topological algebra may benefit from a thorough reading 
of this section.
The basic definition of pseudocompactness should be understood from Section~\ref{sec:WA}. 
The only results from this section that are crucial to later diagonalizability results are
Theorem~\ref{thm:WA.jq} and the implication (e)$\Rightarrow$(d) from Theorem~\ref{thm:WA}, 
the latter of which is established in~\cite[Proposition~IV.12]{G}, or through the
implications (e)$\Rightarrow$(f)$\Rightarrow$(d) using~\cite[Theorem~29.7]{Warner}.
With this background, readers can proceed to Section~\ref{sec:diagonal}. Finally, some readers
may wish to note that Theorem~\ref{thm:duality} is only used in Corollary~\ref{cor:subalg} to 
bound the cardinality of an index set; this bound may be achieved through a more elementary
argument via topological algebra, an exercise which is encouraged for those who are interested.

\subsection{Some conventions}

While diagonalizability of an operator on a vector space depends keenly upon the field $\K$ of
scalars under consideration, we will often suppress explicit reference to $\K$ in notation such
as $\End(V)$ or $\D(V)$ (and similarly for the case of $\End(V)$ when $V$ is a vector space over
a division ring). This does not pose any serious risk of confusion because we work with a single
field of scalars $\K$, with no field extensions in sight. 

Unless otherwise noted, we follow the convention that all rings are associative with a multiplicative identity, 
and all ring homomorphisms preserve the identity. Given a ring $R$, we let $J(R)$ denote its
Jacobson radical. Further, we emphasize that a module $M$ is a left (respectively, right) $R$-module
via the notation ${}_R M$ (respectively, $M_R$).

\section{Topological algebra in $\mathrm{End}(V)$}
\label{sec:topology}

In this section we collect some basic definitions and results from topological algebra that 
will be used in later sections, beginning with a quick review of the most fundamental notions.
We expect that readers are familiar with basic point-set topology, particularly including product
topologies and convergence of nets (as in~\cite[Chapters~2--3]{Kelley}, for instance).

Recall that a \emph{topological abelian group} is an abelian group $G$ equiped with a topology
for which addition and negation form continuous functions $+ \colon G \times G \to G$ and 
$- \colon G \to G$, where $G \times G$ is considered with the product topology.

A \emph{topological ring} is a ring $R$ equipped with a topology such that $R$ forms a 
topological abelian group with respect to its addition, and such that multiplication forms a 
continuous function $R \times R \to R$.
Given a topological ring $R$, a \emph{left topological module} over $R$ is a topological abelian 
group $M$ endowed with a left $R$-module structure, such that $R$-scalar multiplication 
forms a continuous function $R \times M \to M$. 
Right topological modules are similarly defined. 
Any subring or submodule of a topological ring or module may be considered again as a
topological ring or module in its subspace topology. 
The reader is referred to~\cite{Warner} for further background on topological rings 
and modules.

All topologies that we consider in this paper will be Hausdorff. A topological abelian group is 
Hausdorff if and only if the singleton $\{0\}$ is closed for its topology~\cite[Theorem~3.4]{Warner}.
Note that the quotient of a topological ring by an ideal or the quotient of a topological module
by a submodule is again a topological ring or module under the quotient topology (the finest topology
with respect to which the quotient map is continuous); if the ideal
or submodule is closed, then the induced topology on the quotient will consequently be
Hausdorff.

We will also require the notion of completeness for topological rings and modules.
Let $G$ be a Hausdorff topological abelian group, so that any net in $G$ has at most
one limit (see \cite{Kelley}, Theorem~3 of Chapter~2). We say that a net 
$(g_i)_{i \in I}$ in $G$ is \emph{Cauchy} if, for every open
neighborhood $U \subseteq G$ of $0$, there exists $N \in I$ such that,
for all $m,n \geq N$ in the directed set $I$, one has $g_m - g_n \in U$. We say
that $G$ is \emph{complete} if every Cauchy net in $G$ converges. As closed
subsets contain limits of nets, a closed subgroup of a complete topological abelian
group is complete.
(Note that some references such as~\cite[III.3]{Bourbaki:topology} and~\cite[I.7]{Warner} define
completeness in terms of Cauchy \emph{filters} rather than Cauchy nets. However, one can reconcile 
these two apparently different notions as equivalent by noting the connection between convergence
of nets and convergence of filters, as in~\cite[Problem~2.L(f)]{Kelley}.)

A left topological module $M$ over a topological ring $R$ is said to be \emph{linearly
topologized} if it has a neighborhood basis of $0$ consisting of submodules of $M$.
We say that a ring $R$ has a \emph{left linear topology} (or is \emph{left linearly topologized})
if ${}_R R$ is a linearly topologized left $R$-module (i.e., $R$ has a neighborhood basis of 
$0$ consisting of left ideals).
Any subring or submodule of a (left) linearly topologized ring is clearly (left) linearly 
topologized in the induced topology.

\subsection{The finite topology}

Let $X$ be a set and $Y$ a topological space, and let $Y^X$ denote the set of all functions 
$X \to Y$. The \emph{topology of pointwise convergence} on $Y^X$ is the product topology 
under the identification $Y^X = \prod_X Y$.  

If $Y$ has the discrete topology, $Y^X$ has a base of open sets given by sets of the 
following form, for any fixed finite lists of elements  $x_1, \dots, x_n \in X$ and 
$y_1, \dots, y_n \in Y$:
\[
\{f \colon X \to Y \mid f(x_1) = y_1, \dots, f(x_n) = y_n\}.
\]
This is often called the \emph{finite topology} on $Y^X$; as a product of discrete spaces, 
this space is locally compact and Hausdorff.

Now let $V$ be a right vector space over a division ring $D$. Then $\End(V) = \End_D(V) \subseteq V^V$
is a closed subset of $V^V$ (as in the proof of~\cite[Proposition~1.2.1]{DNR}). 
It inherits a topology from the finite topology on $V^V$; we will also refer to this
as the \emph{finite topology} on $\End(V)$.  
Under this topology, $\End(V)$ is a topological ring~\cite[Theorem~29.1]{Warner}.
Alternatively, we may view the finite topology on $\End(V)$ as the left
linear topology generated by the neighborhood
base of $0$ given by the sets of the following form, where $X \subseteq V$
is a finite-dimensional subspace:
\[
X^\perp = \{T \in \End(V) \mid T|_X = 0\}.
\]
These neighborhoods of $0$ coincide with the left annihilators of the form
$\lann(F)$ where $F \in \End(V)$ has finite rank. Furthermore, for every such left ideal
$I = X^\perp$, the left $\End(V)$-module $\End(V)/I$ has finite length.
Indeed, one can readily verify that the map $\End(V) \to \Hom_D(X,V)$ given
by $T \mapsto T|_X$ fits into the short exact sequence of left
$\End(V)$-modules
\[
0 \to X^\perp \to \End(V) \to \Hom_D(X,V) \to 0,
\]
and that $\Hom_D(X,V) \cong V^{\dim(X)}$ is a semisimple left $\End(V)$-module
of finite length.

Because the discrete abelian group $V$ is Hausdorff and complete, so are 
$V^V$ and the closed subspace $\End(V)$. Thus we see that \emph{the topological ring 
$\End(V)$ is left linearly topologized, Hausdorff and complete.}

Throughout this paper, whenever we refer to $\End(V)$ as a topological ring or algebra,
we will always consider it to be equipped with the finite topology. Also, if the underlying
division ring $D = \K$ is a field, then we will consider $V$ as a left $\K$-vector space as in
classical linear algebra.

\begin{remark}
If $R$ is any topological ring and $S$ is a subring of $R$, then its closure $\overline{S}$
in $R$ is readily verified to be a subring of $R$ as well.  For instance, to check
that $\overline{S}$ is closed under multiplication, let $m \colon R \times R \to R$
be the multiplication map. Since $S$ is a subring we have $S \times S \subseteq m^{-1}(S)$, 
and by continuity of $m$ we find that $\overline{S} \times \overline{S} = \overline{S \times S}
\subseteq m^{-1}(\overline{S})$, proving that $\overline{S}$ is closed under multiplication.
\end{remark}

\subsection{Commutative closed subalgebras}

The following lemma shows that in the topological algebra $\End(V)$, commutativity 
behaves well under closure. 
It is surely known in other settings (for part~(1), see~\cite[Theorem~4.4]{Warner}), but we 
include a brief argument for completeness. 
For a subset $X$ of a ring $R$, its  \emph{commutant} (often called the \emph{centralizer}) 
is the subring
\[
X' = \{r \in R \mid rx = xr \mbox{ for all } x \in X\}.
\]

\begin{lemma}
\label{lem:commutativeclosure}
Let $R$ be a Hausdorff topological ring.
\begin{enumerate}
\item The closure of any commutative subring of $R$ is commutative.
\item The commutant of any subset of $R$ is closed in $R$.
\item Any maximal commutative subring of $R$ is closed.
\end{enumerate}
\end{lemma}

\begin{proof}
Define a function $f \colon R \times R \to R$ by $f(x,y) = xy - yx$. Because $R$ is 
a topological ring, $f$ is continuous (where $R \times R$ is given the product topology).  
By the Hausdorff property, the singleton $\{0\}$ is closed in $R$. It follows that $f^{-1}(0)$ 
is closed in $R \times R$.

(1) Suppose that $C \subseteq R$ is a commutative subring. Then $C \times C \subseteq R \times R$
is contained in the closed subset $f^{-1}(0)$. It follows that $\overline{C} \times \overline{C} = 
\overline{C \times C} \subseteq f^{-1}(0)$, whence $\overline{C}$ is commutative.

(2) Let $X \subseteq R$. For each $x \in X$, define $g_x \colon R \to R$ by $g_x(y) = f(x,y)$.
Because $f$ is continuous, so is each $g_x$.
So $X' = \bigcap_{x \in X} g_x^{-1}(0)$ is an intersection of closed sets and is therefore
closed in $R$.

(3) If $C \subseteq R$ is a maximal commutative subring, then $C = \overline{C}$ follows
from either~(1) (by maximality) or from~(2) (as maximality implies $C = C'$).
\end{proof}

If $H$ is a Hilbert space and $B(H)$ is the algebra of bounded operators on $H$, then there is a 
well-known characterization of maximal abelian $*$-subalgebras of $B(H)$ in measure-theoretic
terms~\cite[Theorem~1]{Se}, which can be used to describe their structure quite explicitly (for 
instance, see the proof of~\cite[Lemma~6.7]{He}).
Lest one hope that maximal commutative subalgebras of $\End(V)$ should admit such a simple 
description, Proposition~\ref{prop:discreteclosed} below shows that a full-blown classification of 
such subalgebras is a wild problem.

For the next few results, we restrict to vector spaces $V$ over a field $\K$. Suppose that $A$ is
a subalgebra of $\End(V)$. Then the commutant $A' \subseteq \End(V)$ is the set of those operators
whose action commutes with that of every element of $A$. That is to say, $A'$ is the endomorphism
ring (acting from the left) of $V$ considered as a left $A$-module; in symbols, one might write
$A' = \End_A(V) \subseteq \End_\K(V)$.

\begin{proposition}
\label{prop:discreteclosed}
Let $V$ be a $\K$-vector space and $A$ be a $\K$-algebra with $\dim(A) = \dim(V)$. Then $\End(V)$
contains a closed subalgebra (even a commutant) that is topologically isomorphic to $A$ with the
discrete topology. If $A$ is commutative, then this subalgebra may furthermore be chosen to be a
maximal commutative subalgebra.
\end{proposition}

\begin{proof}
If $\phi:A\rightarrow V$ is a vector space isomorphism, we can identify $\End(A)=\End_\K(A)$ with 
$\End_\K(V)$, and we may prove everything in $\End(A)$. The left regular representation 
$\lambda:A\rightarrow \End(A)$, where $\lambda(a) \colon x\mapsto ax$, and the right regular 
representation $\rho:A\rightarrow \End(A)$, where $\rho(a) \colon x\mapsto xa$, respectively 
define a homomorphism and an anti-homomorphism of algebras. Certainly $\rho$ and $\lambda$ 
are injective because $\rho(a)$ and $\lambda(a)$ both send $1 \mapsto a$; furthermore, the 
intersection of the open set $\{T \in \End(A) \mid T(1)=a\}$ with either $\lambda(A)$ or $\rho(A)$ 
is the singleton $\{a\}$. Therefore, the induced topologies from $\End(A)$ on the algebras 
$\lambda(A)$ and $\rho(A)$ are both discrete, and the maps $\lambda$ and $\rho$ are continuous.

By the comments preceding this proposition, the commutants $\lambda(A)'$ and $\rho(A)'$ are 
respectively the endomorphism rings of $A$ considered as a left and right module over itself. 
It is well known that these endomorphism rings are $\lambda(A)' = \rho(A)$ (that is, left 
$A$-module endomorphisms of $A$ are right-multiplication operators) and $\rho(A)' = \lambda(A)$; 
for instance, see~\cite[Example~1.12]{FC}.  Thus $\lambda(A)'' = \rho(A)' = \lambda(A)$, and
$\lambda(A)$ is closed in $\End(V)$ thanks to Lemma~\ref{lem:commutativeclosure}(2). Consequently,
$\lambda(A)$ a discrete closed subalgebra of $\End(A)$. 

Finally, note that in the case where $A$ is commutative we have $\lambda(a) = \rho(a)$ for all 
$a \in A$. So $\lambda(A)' = \rho(A) = \lambda(A)$, proving that the image of $A$ in $\End(V)$ is a
maximal commutative subalgebra.  
\end{proof}

The property of being a discrete subalgebra of $\End(V)$ translates nicely into 
representation-theoretic terms. First, having a subalgebra $A\subseteq \End(V)$ is equivalent 
to saying that $V$ is a faithful $A$-module. Faithfulness of $V$ is equivalent to the existence of 
a set $\{v_i \mid i \in I\} \subseteq V$ such that, for any $a \in A$, if $av_i = 0$ for all $i$ then
$a = 0$. But the existence of such a set is further equivalent to the existence of a left $A$-module
embedding $A \hookrightarrow V^I$, given by $a \mapsto (av_i)_{i \in I}$.
The next proposition shows that being a discrete subalgebra amounts the requirement that 
this embedding is into a finite power of $V$.

\begin{proposition}
\label{prop:discrete representation}
Let $A\subseteq \End(V)$ be a subalgebra. Then $A$ is discrete if and only if there is a left $A$-module
embedding of $A$ into $V^n$ for some positive integer $n$. 
\end{proposition} 
\begin{proof}
The subalgebra $A$ is discrete if and only if $\{0\}$ is open in $A$, which is further equivalent 
to  $\{0\} = A\cap \{T \in \End(V) \mid T(W)=0\}$, for some finite-dimensional 
subspace $W$ of $V$.  Thus if $A$ is discrete then there are $w_1,\dots,w_n\in V$ such that 
$a w_i=0$ ($i=1,\dots,n$) implies $a=0$ for all $a \in A$, so that that the map $\varphi \colon
A \to V^n$ given by $\varphi(a)=(aw_i)_{i}\in V^n$ is an injective morphism of $A$-modules.  
Conversely, assume that there is an injective $A$-module homomorphism $\phi \colon A \to V^n$,
and denote $\phi(1) = (w_1, \dots, w_n)$. Then $aw_i = 0$ implies $a = 0$ for all $a \in A$
thanks to injectivity of $\phi$, which translates as above to $A$ being a discrete subalgebra of
$\End(V)$.
\end{proof}

\begin{corollary}\label{cor:discrete closed}
Let $V$ be an infinite-dimensional $\K$-vector space and $A$ a $\K$-algebra. Then $\End(V)$ contains a closed subalgebra that is topologically isomorphic to $A$ with the discrete topology if and only if $\dim(A) \leq \dim(V)$.
\end{corollary}

\begin{proof}
Suppose that $\dim(A) \leq \dim(V)$, and let $B$ be any $\K$-algebra with $\dim(B) = \dim(V)$. Then $\dim(A \otimes B) = \dim(A) \dim(B) = \dim(B)$, since $\dim(B) = \dim(V)$ is infinite. So, by Proposition~\ref{prop:discreteclosed}, $\End(V)$ has closed subalgebra that is topologically isomorphic to $A \otimes B$ with the discrete topology, which in turn contains the closed discrete subalgebra $A \otimes 1 \cong A$.

Converely, if $A \subseteq \End(V)$ is a discrete subalgebra, then $A\hookrightarrow V^n$ with $n$
finite by Proposition~\ref{prop:discrete representation}. This shows that $\dim(A)\leq n\cdot \dim(V)
= \dim(V)$ since $\dim(V)$ is infinite.
\end{proof}

We use the above to illustrate a few examples of discrete maximal commutative subalgebras
of $\End(V)$ generated by ``shift operators'' when $V$ has countably infinite dimension.

\begin{example}\label{ex:right shift}
Let $A = \K[x]$ be the polynomial algebra in a single indeterminate. Let $V$ be a $\K$-vector 
space with basis $\{v_0, v_1, v_2, \dots\}$. Then Proposition~\ref{prop:discreteclosed} implies
that $\End(V)$ has a discrete maximal commutative subalgebra isomorphic to $A$.
Under the $\K$-linear isomorphism $A \cong V$ that sends $x^i \mapsto v_i$, the
embedding $A \to \End(V)$ of Proposition~\ref{prop:discreteclosed} sends $x$ to the 
right-shift operator $S \in \End(V)$, with
\[
S(v_i) = v_{i+1} \quad \mbox{for all } i \geq 0.
\]
\end{example}

\begin{example}
Let $A = \K[x,x^{-1}]$ be the Laurent polynomial ring over $\K$. If $V$ is a $\K$-vector
space with basis $\{v_i \mid i \in \Z\}$, then the vector space isomorphism $A \cong V$
sending $x^i \mapsto v_i$ gives an embedding $A \to \End(V)$ of $A$ onto a 
discrete maximal commutative subalgebra of $\End(V)$. One can check that the embedding
provided in the proof of Proposition~\ref{prop:discreteclosed} sends $x \in A$ to the invertible ``infinite shift'' 
operator $T \in \End(V)$, with
\[
T(v_i) = v_{i+1} \quad \mbox{for all } i \in \Z.
\]
\end{example}

\subsection{Summability of idempotents}

Another topic that will play a role in diagonalizability of subalgebras is the ability to form
the ``sum'' of an infinite set of orthogonal idempotents. 

\begin{definition}
\cite[Section 10]{Warner}
Let $G$ be a Hausdorff topological abelian group, and let $\{g_i \mid i \in I\} \subseteq G$.
If the net of finite sums of the $g_i$, indexed by finite subsets of $I$, converges to a limit in $G$, 
then we write this limit as $\sum g_i$, and we say for brevity that \emph{$\sum g_i$ exists (in $G$)} 
or that the set $\{g_i\}$ is \emph{summable}.
\end{definition}

Perhaps the most basic nontrivial example of a summable set of elements is as follows:
let $I$ be a set, and consider the product space $\K^I$ as a topological $\K$-algebra. 
Let $e_i \in \K^I$ denote the idempotent whose $i$th coordinate is $1$ and whose other
coordinates are $0$. Then for any $\lambda_i \in \K$, the set $\{\lambda_i e_i\}$ is
summable; in fact, it is clear that each element $x \in \K^I$ has a unique 
expression of the form $x = \sum \lambda_i e_i$.

Recall that there are three preorderings for idempotents in a ring $R$: given idempotents
$e,f \in R$, one defines
\begin{alignat*}{2}
e \leql f &\iff ef = e & &\iff Re \subseteq Rf \\
e \leqr f &\iff fe = e & &\iff eR \subseteq fR \\
e \leq f &\iff ef = e =fe & &\iff eRe \subseteq fRf.
\end{alignat*}
While the last of these is a partial ordering, the first two generally are not. (For instance, the
idempotents $e = \left(\begin{smallmatrix}1 & 0 \\ 0 & 0\end{smallmatrix}\right)$ and 
$f = \left( \begin{smallmatrix}1 & 0 \\ 1  & 0\end{smallmatrix}\right)$ in $\M_2(\mathbb{Q})$ satisfy 
$e \leql f \leql e$ with $e \neq f$.) 

\begin{lemma}
\label{lem:lubsums}
Let $R$ be a Hausdorff topological ring with an orthogonal set of idempotents $\{e_i \mid i \in I\}
\subseteq R$. Suppose that $\sum e_i$ exists. 
Then $\sum e_i$ is an idempotent that is the least upper bound for the set $\{e_i\}$ in the 
preorderings $\leql$ and $\leqr$ and the partial ordering $\leq$.
\end{lemma}

\begin{proof}
Write $e = \sum e_i$.
Given a finite subset $J \subseteq I$, write $e_J = \sum_{i \in J} e_i$, so $e$ is the limit of the
net $(e_J)$. Because each $e_J$ lies in the zero set of the 
continuous function $R \to R$ given by $x \mapsto x-x^2$, it follows that the limit $e$ also 
lies in this set. So $e = e^2$ is idempotent. 
Further, $e_k\leq e$  for all $k \in I$ (and $e_k\leql e$, $e_k\leqr e$) because for any finite 
$J \subseteq I$ with $J\supseteq \{k\}$, $e_ke_J=e_Je_k=e_k$, and the continuity of the functions 
$x\mapsto e_kx$, $x\mapsto xe_k$ shows that the relation holds in the limit as well: $e_ke=ee_k=e$.
Thus $e$ is an upper bound for the idempotents $e_i$ with respect to $\leql$ and $\leqr$ (and
consequently, $\leq$).

To see that $e$ is the least upper bound of the $e_i$, let $f \in R$ be idempotent, and suppose 
that $e_i \leql f$ for all $i$. Then for all finite
$J \subseteq I$, we have $e_J f = \sum_{j \in J} e_j f = \sum_{j \in J} e_j = e_J$. Thus all $e_J$
lie in the closed subset $\lann(1-f) \subseteq R$, so we also have $e = \lim e_J \in \lann(1-f)$,
which translates to $e \leql f$. Similarly, if $f = f^2 \in R$ satisfies $e_i \leqr f$ (respectively,
$e_i \leq f$) for all $i$, then $e \leqr f$ (respectively, $e \leq f$). So $e$ is a least
upper bound for the $e_i$ in all of these (pre)orderings. 
\end{proof}

Suppose that $\{E_i \mid i \in I\}$ is an orthogonal set of idempotents in $\End(V)$. Notice
that the sum of subspaces $\sum \range(E_i) \subseteq V$ is direct. Indeed, if $\sum r_i = 0$
in $V$ where $r_i \in \range(E_i)$ are almost all zero, then for every $j \in I$ we have 
$0 = E_j(\sum r_i) = E_j(\sum E_i(r_i)) = E_j(r_j) = r_j$ as desired. Similarly, it is straightforward
to show that the following sum of subspaces in $V$ is direct:
\begin{equation}
\label{eq:directsum}
\left(\bigoplus \range(E_i) \right) \oplus \left(\bigcap \ker(E_i)\right) \subseteq V.
\end{equation}
Whether or not this direct sum is equal to the whole space $V$ strictly controls the existence
of $\sum E_i$ in $\End(V)$.

\begin{lemma}
\label{lem:sumconditions}
Let $R= \End(V)$, and let $\{E_i \mid i \in I\} \subseteq R$ be an orthogonal set of
idempotents. The following are equivalent:
\begin{enumerate}[label=(\alph*)]
\item $\sum E_i$ exists in $R$;
\item $V = \left( \bigoplus \range(E_i) \right) \oplus \left( \bigcap \ker(E_i) \right)$ (i.e., the
containment \eqref{eq:directsum} is an equality);
\item For every $v \in V$, the set $\{i \in I \mid v \notin \ker(E_i)\}$ is finite;
\item For every finite-dimensional subspace $W \subseteq V$, there exists a finite subset 
$J \subseteq I$ such that $(1-\sum_{j \in J} E_j)W \subseteq \bigcap \ker(E_i)$.
\end{enumerate}
When the above conditions hold, then $\{E_i \mid i \in I_0\}$ is also summable for any subset 
$I_0 \subseteq I$, and $\sum_{i \in I} T_i$ exists for any elements $T_i \in E_i R E_i$.
\end{lemma}

\begin{proof}
First assume that~(a) holds; we will verify~(d). Let $W \subseteq V$ be any finite-dimensional 
subspace. Denote $E = \sum E_i$, the limit of the finite sums of the $E_i$. Because
$\{T \in \End(V) \mid T|_W = E|_W\}$ is an open neighborhood of $E$ in $\End(V)$, there exists
a finite subset $J \subseteq I$ such that $\sum_{j \in J} E_j$ lies in this set. Then for every $i \in I$,
since $E_i \leq E$ gives $E_i(1-E) = 0$, we conclude that $E_i(1-\sum_{j \in J}E_j)W = E_i(1-E)W = 0$.

Now assume~(d) holds. To prove~(c), let $v \in V$ and consider the subspace $W = \Span(v)$ of $V$. 
Then there exists a finite subset $J \subseteq I$ such that, for $E_J = \sum_{j \in J} E_j$ we
have  $(1- E_J)W \subseteq \bigcap \ker(E_i)$. Thus for all $i \in I$, $E_i(1-E_J)(v) = 0$, from which it
follows that
\[
E_i(v) = E_i \sum_{j \in J} E_j(v)  = \sum_{j \in J} E_i E_j(v).
\]
The last sum is $0$ whenever $i \notin J$, and so $\{i \in I \mid v \notin \ker(E_i)\} \subseteq J$ is finite, as desired.

Next suppose that~(c) holds; we will verify~(b). Let $v \in V$, so that the set
$S = \{i \in I \mid  v \notin \ker(e_i)\}$ is finite by hypothesis. Then $E_i v = 0 $ for all 
$i \notin S$. With this, one can readily verify that $v = (\sum_{i \in S} E_i)v + (1 - \sum_{i \in S}E_i)v$ 
is a sum of elements in $\bigoplus \range(E_i)$ and $\bigcap \ker(E_i)$, respectively.
So equality indeed holds in~\eqref{eq:directsum}.

Now assume that~(b) holds, and let $T_i \in E_i R E_i$ for all $i$. We will prove that $\sum T_i$
exists in $R$; this will imply that~(a) holds (the particular case where all $T_i = E_i$).
Define $T \in R$ to be the linear operator on $V$ with kernel $L =  \bigcap \ker(E_i)$ whose 
restriction to each $U_i = E_i(V)$ agrees with $T_i$. A basic open neighborhood of $T$ in $R$ has the 
form $\{S \in R \mid S|_X = T|_X\}$ for some finite set $X \subseteq V$. There is a finite set $J \subseteq I$ 
such that  $X \subseteq  \left(\bigoplus_{j \in J} U_j\right) \oplus L$. It follows that for any finite 
subset $J' \subseteq I$  with $J' \supseteq J$, the operators $T$ and $\sum_{i \in J'} T_i$ 
have the same restrictions to $\left(\bigoplus_{j \in J} U_j\right) \oplus L \supseteq X$. Thus the net 
of finite sums of the $T_i$ converges to $T$.

Finally, to see that $\sum_{i \in I_0} E_i$ exists for any $I_0 \subseteq I$, notice that
condition~(c) still holds when $I$ is replaced with $I_0$.
\end{proof}

In the proof above, the idempotent sum $\sum E_i$ is explicitly described as the projection of 
$V$ onto the subspace $\bigoplus \range(E_i)$ with kernel equal to $\bigcap \ker(E_i)$.

\begin{corollary}
\label{cor:sumconditions}
For an orthogonal set of idempotents $\{E_i \mid i \in I\} \subseteq \End(V)$, the following
are equivalent:
\begin{enumerate}[label=(\alph*)]
\item $\{E_i\}$ is summable and $\sum E_i = 1 \in \End(V)$;
\item $\bigoplus \range(E_i) = V$;
\item For every finite-dimensional subspace $W \subseteq V$, there exists a finite subset 
$J \subseteq I$ such that $\sum_{j \in J} E_j$ restricts to the identity on $W$.
\end{enumerate}
\end{corollary}

\begin{proof}
This follows from a straightforward translation of conditions (a), (b), and (d) in
Lemma~\ref{lem:sumconditions}, realizing that if $\{E_i\}$ is summable, then 
$\sum E_i = 1$ if and only if $\bigcap \ker(E_i) = \{0\}$.
\end{proof}

The following example shows that an arbitrary set of orthogonal idempotents in $\End(V)$
need not be summable.

\begin{example}\label{ex:not summable}
Let $V$ be a $\K$-vector space with countably infinite basis $\{v_0, v_1, \dots\}$.
Consider the $\K$-algebra $A$ generated by countably many idempotents $x_1, x_2, \dots \in A$
subject to the conditions $x_i x_j = 0$ when $i \neq j$. Then $A$ has $\K$-basis given by
$\{1,x_1,x_2,\dots\}$. Under the isomorphism (identification) $\phi \colon A \to V$ given by $\phi(1) = v_0$
and $\phi(x_i) = v_i$ for $i \geq 1$, the proof of Proposition~\ref{prop:discreteclosed} provides an
injective homomorphism $\lambda \colon A \to \End(V)$ such that $\lambda(A)$ is a closed
discrete maximal commutative subalgebra of $\End(V)$. 
Denote each $\lambda(x_i) = E_i \in \End(V)$. 
(If we let $\{E_{ij} \mid i,j \geq 0\}$ denote the infinite set of matrix units of $\End(V)$ with 
respect to the basis $\{v_i \mid i \geq 0\}$, then we may explicitly describe $E_i = E_{i0} + E_{ii}$.)
One may quickly deduce from Lemma~\ref{lem:sumconditions}(c) (with $v = v_0$) that $\sum E_i$ 
does not exist in $\End(V)$. 
\end{example}

\begin{lemma}
\label{lem:sumofproducts}
Suppose that $\{E_i \mid i \in I\}$ and $\{F_j \mid j \in J\}$ are orthogonal sets of idempotents in
$\End(V)$ such that $E = \sum E_i$ and $F = \sum F_j$ exist.  Furthermore, suppose that all of 
the $E_i$  and $F_j$ commute pairwise. Then the orthogonal set of idempotents 
$\{E_i F_j \mid (i,j) \in I \times J\}$ is also summable, and $\sum_{i,j} E_i F_j = EF$.
\end{lemma}

\begin{proof}
We begin by remarking that $E_i F_j = F_j E_i$ for all $i$ and $j$ implies that $E$ and $F$ commute
with each other and with each $E_i$ and $F_j$.

Fix a finite-dimensional subspace $W \subseteq V$. Then there exist finite sets $I_0 \subseteq I$
and $J_0 \subseteq J$ such that $E -\sum_{i \in I_0} E_i$ and $F - \sum_{j \in J_0} F_j$ both restrict
to zero on $W$. Then
\begin{align*}
EF - \sum\nolimits_{I_0 \times J_0} E_i F_j 
&= EF - \left(\sum\nolimits_{I_0} E_i \right)F + \left(\sum\nolimits_{I_0} E_i \right)F
  - \left(\sum\nolimits_{I_0} E_i\right) \left(\sum\nolimits_{J_0} F_j\right) \\
&= F\left(E - \sum\nolimits_{I_0} E_i \right) 
  + \sum\nolimits_{I_0} E_i \left( F - \sum\nolimits_{J_0} F_j \right).
\end{align*}
The latter expression also restricts to zero on $W$, and since $W$ was arbitrary, we conclude
that $\sum E_i F_j$ exists and is equal to $EF$.
\end{proof}

\subsection{Topological modules}

In the remainder of this section, we record some results relating to topological modules
that will be of use when considering pseudocompact modules. 
The first is a link between open submodules and discrete modules.

\begin{lemma}\label{lem:open}
Let $M$ be a topological module over a topological ring $R$, and let
$N$ be a submodule of $M$. Then the following are equivalent:
\begin{enumerate}[label=(\alph*)]
\item $N$ is open;
\item $M/N$ is discrete in the quotient topology;
\item $N$ is the kernel of a continuous homomorphism from $M$ to
a discrete topological $R$-module.
\end{enumerate}
Furthermore, any open submodule of $M$ is closed, and any submodule of 
$M$ containing an open submodule is itself open.
\end{lemma}

\begin{proof}
That (b) implies (c) follows from the fact that the canonical surjection
$M \to M/N$ is continuous when $M/N$ is equipped with the quotient
topology. To see that~(c) implies~(a), simply consider a continuous module homomorphism
$f \colon M \to D$ as in~(c) and note that $N = f^{-1}(0)$ where $\{0\}$ is open
in $D$.
Finally, suppose~(a) holds and let $\pi \colon M \to M/N$ be the canonical surjection.
For any coset $x+N \in M/N$, one has $\pi^{-1}(x+N) = x+N \subseteq M$. As $N$
is open, the same is true of the translate $x+N$. It follows that each singleton of $M/N$
is open, whence it is discrete.

That any open submodule of $M$ is closed follows, for instance, from characterization~(c)
above, as the kernel of a continuous homomorphism to a Hausdorff module is closed.
For the final claim, assume that $N_0 \subseteq N \subseteq M$ are submodules
where $N_0$ is open. One may readily verify that $N$ is open as it is the kernel
of the composite $M \to M/N_0 \to (M/N_0)/(N/N_0) \cong M/N$ of surjective continuous 
homomorphisms, where the latter two modules are discrete.
\end{proof}

The next result characterizes modules that are inverse limits
of discrete modules. We first make some general remarks on inverse limits of topological
modules. (These remarks also apply to limits of general diagrams, but we restrict to inversely 
directed systems for notational simplicity as these are the only systems we require.)
Given a topological ring $R$, let ${}_R \TMod$ denote the category of left topological $R$-modules 
with continuous module homomorphisms. Given an inversely directed system $\{M_j \mid j \in J\}$ 
with connecting morphisms $\{f_{ij} \colon M_j \to M_i \mid i \leq j \mbox{ in } J\}$ in ${}_R \TMod$, 
its limit can be constructed via the usual ``product-equalizer'' method as in~\cite[V.2]{MacLane}:
\[
\varprojlim M_j = \{(m_j) \in \prod M_j \mid f_{ij}(m_j) = m_i \mbox{ for all } i \leq j \mbox{ in }J\} \subseteq \prod M_j.
\]
(The same product-equalizer construction yields both the inverse limit of the $M_j$ as 
$R$-modules and the inverse limit as topological groups.) The topology on the inverse limit
is the \emph{initial topology} with respect to the canonical morphisms 
$g_i \colon \varprojlim M_j \to M_i$ for all $i \in J$; this is the topology generated by the 
subbasis of sets of the form $g_i^{-1}(U)$ for any $i \in J$ and open $U \subseteq M_i$.
(In fact, thanks to the inverse directedness of the system $M_j$, the sets of the form
$g_i^{-1}(U)$ are closed under intersection and actually form a basis for the topology: given
$i,j \leq k$ in $J$ and open $U \subseteq M_i$ and $V \subseteq M_j$, we have
$g_i^{-1}(U) \cap g_j^{-1}(V) = g_k^{-1}(f_{ik}^{-1}(U) \cap f_{jk}^{-1}(V))$.)

Suppose that all of the $M_j$ in the inverse system above are Hausdorff. Then the product 
of any subset of the $M_j$ will also be Hausdorff. 
Since the above presentation of the limit can be viewed as the equalizer (i.e., kernel of the difference)
of two continuous maps $\prod M_j \to \prod_{f_{ij}} M_i$ (see~\cite[Theorem V.2.2]{MacLane}),
and the equalizer of morphisms between Hausdorff spaces is closed in the domain, we see that
$\varprojlim M_j$ forms a closed submodule of $\prod M_j$. (See also~\cite[III.7.2]{Bourbaki:topology}.)

Because the universal property of an (inverse) limit $\varprojlim M_j$ identifies it within
${}_R \TMod$ uniquely up to a unique isomorphism, when a topological module $L$ is
isomorphic to such a limit (i.e., satisfies the universal property) then we will write 
$L = \varprojlim M_j$ without danger of confusion.

\begin{remark}\label{rem:inverse limit}
Suppose that $\mathcal{P}$ is any property of Hausdorff topological modules that
is preserved when passing to products and closed submodules of modules satisfying $\mathcal{P}$.
Then the (inverse) limit of any system of Hausdorff modules satisfying $\mathcal{P}$ will
again have property $\mathcal{P}$ (as it is a closed submodule of a product of modules
satisfying $\mathcal{P}$ thanks to the discussion above). In particular, this holds when
$\mathcal{P}$ is taken to be either of the properties of being \emph{complete} or
\emph{linearly topologized}.
\end{remark}

We are now ready to characterize inverse limits of discrete modules.
(Note that special cases of the following were given in~\cite[Theorem~3]{Z1}
and~\cite[Corollary~5.22]{Warner}.)

\begin{lemma}\label{lem:prodiscrete}
For a left topological module $M$ over a topological ring $R$, the following are equivalent:
\begin{enumerate}[label=(\alph*)]
\item $M$ is Hausdorff, complete, and linearly topologized;
\item $M= \varprojlim M/N$ in the category ${}_R \TMod$, where $N$ ranges over any 
neighborhood basis $\U$ of $0$ consisting of open submodules of $M$ (and the connecting
homomorphisms for $N \supseteq N'$ are the canonical surjections $M/N' \twoheadrightarrow
(M/N')/(N/N') \cong M/N$);
\item $M$ is an inverse limit of discrete topological $R$-modules. 
\end{enumerate}
When $M$ satisfies the above conditions and $\U$ is any neighborhood basis
of $0$ consisting of open submodules, the topology on $M$ is induced from the product 
topology of $\prod_{N \in \U} M/N$ via the usual inclusion of the inverse limit.
\end{lemma}

\begin{proof}
Suppose that~(a) holds and let $\U$ be a neighborhood basis of $0$ consisting of open
submodules of $M$. To deduce~(b), first note that the Hausdorff property of $M$ ensures that
natural map $\phi \colon M \to \prod_{N \in \U} M/N$ is an embedding. The image of this map lies 
in the closed submodule  $\varprojlim_{N \in \U} M/N$ of $\prod M/N$. To see that $\phi$ is surjective,
let $x \in \varprojlim_{N \in \U} M/N$ be represented by the compatible family $(x_N)_{N \in \U}$ with 
$x_N \in M/N$ for all $N \in \U$. For each $N \in \U$, fix $y_N \in M$ such that $y_N + N = x_N$. 
Then it is straightforward to show from the compatibility condition on the $x_N$ that $(y_N)_{N \in \U}$ 
forms a Cauchy net, which converges to some $y \in M$ by completeness. 
We claim that this element satisfies $y + N = y_N + N = x_N$ for all $N \in \U$. Indeed, as 
$y = \lim y_N$, for fixed $N \in \U$ there exists $N' \subseteq N$ such that $y - y_{N'} \in N$.
But compatibility of the family $(x_N) = (y_N + N)$ implies that $y_{N'} + N = y_N + N$, so that
$y + N = y_{N'} +N = y_N + N$. It follows that $y$ has image $\phi(y) = (y + N)_{N \in \U} 
= (x_N)_{N \in \U} = x$, as desired.
Finally, to see that $\phi$ is a
homeomorphism onto the inverse limit, it suffices to show that upon identifying $M$ with
the inverse limit, they share a common neighborhood basis of $0$. As each of the components
of the product $\prod_{N \in \U} M/N$ is discrete by Lemma~\ref{lem:open}, the comments 
preceding this lemma imply that a neighborhood basis of $\varprojlim M/N$ is given by 
$\pi_N^{-1}(0)$, where each $\pi_N$ is the projection of the product onto the corresponding 
factor for $N \in \U$.
But under the isomorphism $M \cong \varprojlim M/N$, each $\pi_N^{-1}(0)$ corresponds
to $N \in \U$, so the claim is proved. Moreover, this establishes the last claim about the
topology on $M$ being induced by that of $\prod_{N \in \U} M/N$.

Clearly (b) implies (c). Finally, assume that (c) holds, so that $M = \varprojlim M_i$ is the limit
of a directed system of discrete topological modules $M_i$. Because the $M_i$ are Hausdorff,
complete, and linearly topologized, we see from Remark~\ref{rem:inverse limit} that the limit 
$M = \varprojlim M_i$ satisfies condition~(a).
\end{proof}

\begin{definition}
A topological module that satisfies the equivalent conditions of the previous lemma will
be called \emph{pro-discrete}.  A topological ring $R$ is \emph{left pro-discrete} if ${}_R R$
is pro-discrete as a topological $R$-module, and right pro-discrete rings are similarly
defined. Further, $R$ is said to be \emph{pro-discrete} if it Hausdorff, complete, and has
a neighborhood basis of zero consisting of two-sided ideals.
\end{definition}

\begin{remark}\label{rem:left and right prodiscrete}
As the terminology suggests, a topological ring $R$ is left and right pro-discrete if and only 
if it is pro-discrete. Clearly every pro-discrete ring is left and right pro-discrete. 
Conversely, suppose that $R$ is left and right pro-discrete, and let $N$ be an open left ideal.
Since $R$ is right pro-discrete, there is an open right ideal $M$ such that $M\subseteq N$, 
and again, by left pro-discreteness there is an open left ideal $N'$ with $N'\subseteq M$. 
Then $I=RM \subseteq N$ since $N$ is a left ideal. Now $I$ contains $RN'=N'$, so that $I$ is
open by Lemma~\ref{lem:open}.
Hence, every open neighborhood of $0$ contains an open two-sided ideal $I$, making $R$
pro-discrete.
\end{remark}

 In the next lemma we note that the property of pro-discreteness is inherited by closed subrings.
(A similar statement holds for topological modules, but we will not make use of this fact.)

\begin{lemma}\label{lem:closed prodiscrete}
Every closed subring of a (left) pro-discrete ring is (left) pro-discrete.
\end{lemma}

\begin{proof}
Suppose $B$ a is (left) pro-discrete ring with a closed subring $A \subseteq B$. 
Then $A$ is certainly Hausdorff in its inherited topology, and because $A$ is closed in
the complete ring $B$ we find that $A$ is also complete. Finally, any open (left) ideal
$I$ of $B$ intersects to an open (left) ideal $I \cap A$ of $A$. Because such $I$ form
a neighborhood base of zero in $B$, the contractions $I \cap A$ form a neighborhood
base of zero in $A$. Thus $A$ is (left) linearly topologized, showing that it is (left)
pro-discrete.
\end{proof}

In particular, our discussion of the finite topology above makes it clear that the topological ring 
$\End(V)$ for a right $D$-vector space $V$ is left pro-discrete. Thus every closed subring of
$\End(V)$ is left pro-discrete. 

In the case when $D = \K$ is a field, this raises an interesting question about representations of 
topological algebras. Fixing an infinite-dimensional $\K$-vector space $V$, to what extent can one 
characterize those left pro-discrete $\K$-algebras that can be realized as closed subalgebras of 
$\End(V)$?
Corollary~\ref{cor:discrete closed} characterizes exactly which \emph{discrete} algebras have
such a representation, purely in terms of their dimension. A topological characterization would
necessarily extend this result.

\section{Infinite Wedderburn-Artin Theorem}
\label{sec:WA}

As mentioned in the introduction, the structure theory of artinian rings plays an important
role in the theory of diagonalizability for operators on a finite-dimensional vector space $V$.
Every subalgebra of $\End(V)$ is finite-dimensional and therefore artinian; diagonalizable
subalgebras are furthermore semisimple. 

If $V$ is infinite-dimensional, we have seen in Proposition~\ref{prop:discreteclosed} that
$\End(V)$ contains a wild array of discrete subalgebras. But it turns out that $\End(V)$ itself,
as well as its diagonalizable subalgebras, satisfy a well-known condition of being ``almost left
artinian'' (in a topological sense) called \emph{pseudocompactness}. In fact, they satisfy an
even stronger ``almost finite-dimensional'' property that we shall call \emph{$\K$-pseudocompactness}.
Further, both of these algebras of interest are Jacobson semisimple, so that they satisfy
a suitable infinite-dimensional version of semisimplicity. 

In this section we aim to present a Wedderburn-Artin theorem for left pseudocompact, Jacobson
semisimple rings in Theorem~\ref{thm:WA}, which gathers some known and somewhat independent 
results on topological semisimplicity along with some new ones. Using our methods, we recover 
some results from~\cite{Z1,Z}, along with the classical Wedderburn-Artin
theorem and similar types of theorems for algebras and coalgebras. 

\subsection{Pseudocompact and linearly compact modules and rings}

Pseudocompactness is an important property in topological algebra that expresses a particular way
for a module or ring to be ``close to having finite length.'' We recall the definition after giving a
few equivalent formulations of this property.

We say that a submodule $N$ of a module $M$ has \emph{finite colength} if $M/N$ is a module of finite length.

\begin{lemma}\label{lem:pseudocompact}
Let $M$ be a left topological module over a topological ring $R$. The following are equivalent:
\begin{enumerate}[label=(\alph*)]
\item $M$ is Hausdorff, complete, and has a neighborhood basis of $0$ consisting of open submodules
of finite colength;
\item $M$ is pro-discrete and every open submodule of $M$ has finite colength;
\item $M$ is an inverse limit in ${}_R \TMod$ of discrete topological $R$-modules of finite length.
\end{enumerate}
\end{lemma}

\begin{proof}
The equivalence of (a) and (c) and the implication (b)$\Rightarrow$(c) follow directly from 
Lemma~\ref{lem:prodiscrete} using the basis $\U$ of open submodules of $M$ that have finite colength. 
Now assume that~(a) holds; we will deduce~(b). It follows from Lemma~\ref{lem:prodiscrete} 
that $M$ is pro-discrete. Given any open submodule $N$ of $M$, we are given that there exists
an open submodule $L \subseteq N$ of finite colength. It follows immediately that $N$ has finite
colength, as desired.
\end{proof}

A module satisfying the equivalent conditions above is called a \emph{pseudocompact} module. A topological ring
$R$ is said to be \emph{left pseudocompact} if ${}_R R$ is a pseudocompact topological module.
We also recall that a complete Hausdorff topological ring $R$ which has a basis of neighborhoods 
of $0$ consisting of two-sided ideals $I$ such that $R/I$ has finite length both on the left and on
the right (i.e., $R/I$ is a a two-sided artinian ring) is called a \emph{pseudocompact ring}. 

\begin{remark}\label{rem:left and right pseudocompact}
Notice that a topological ring $R$ is left and right pseudocompact if and only if it is pseudocompact.
The argument is identical to the one given in Remark~\ref{rem:left and right prodiscrete}, taking into
account that if the one-sided ideals $M,N,N'$ there have finite colength, then the ideal $I$ has finite
colength as both a left and a right ideal, making $R/I$ an artinian ring.
\end{remark}

The following key example is of particular interest to us.

\begin{example}
\label{e.WA}
Let $D$ be a division ring and $V$ a right $D$-vector space. 
As shown in the discussion of the finite topology in Section~\ref{sec:topology}, the topological
ring $\End(V)$ is Hausdorff, complete, and has a basis of open left ideals having finite
colength. Thus $\End(V)$ is left pseudocompact.
\end{example}

\begin{remark}\label{rem:WA.not}
We also observe that in the example above, if $V$ has infinite dimension over $D$, 
then $\End(V)$ is not right pro-discrete and thus is not right pseudocompact. 
Indeed, if $\End(V)$ is pro-discrete, then it has a neighborhood 
basis of $0$ consisting of two-sided ideals. It is well known that the left socle $\Sigma$ 
of $\End(V)$ (the sum of all the simple left ideals of $\End(V)$) is equal to the
ideal of finite rank operators on $V$ (see~\cite[Exercise~11.8]{FC}), 
and that $\Sigma$ is the smallest nonzero two-sided ideal of $\End(V)$ (see~\cite[Exercise~3.16]{FC}). 
Hence $\Sigma \subseteq I$ for every non-zero open ideal $I$, and since $\End(V)$ is 
Hausdorff, it follows that $0$ must be open for $\End(V)$ to be pro-discrete. 
Thus there is a finite-dimensional subspace $W\subseteq V$ with $W^\perp = 0$, 
which shows that $V=W$ must be finite-dimensional.
\end{remark}

Given a topological ring $R$, let ${}_R \PC$ denote the full subcategory of ${}_R \TMod$
whose objects are the left pseudocompact topological $R$-modules. For the portions
of our Wedderburn-Artin theorem that mimic the structure of module categories over
semisimple rings (e.g., every short exact sequence splits), we will make use of Gabriel 
duality for categories of pseudocompact modules. 
To this end, we recall that a \emph{Grothendieck category} is
an abelian category with a \emph{generator} (i.e., an object $G$ such that, for any pair
of morphisms $f,g \colon X \to Y$ with the same domain and codomain, if $f \neq g$
then there exists $h \colon G \to X$ such that $f \circ h \neq g \circ h$) in which all 
small coproducts exist, and with exact direct limits. Also, an abelian category is said to 
be \emph{locally finite} if every object is the (directed) colimit of its finite-length subobjects.
We recommend~\cite[Chapters IV--V]{Stenstroem} as a basic reference on abelian and Grothendieck 
categories. 

\begin{remark}\label{rem:general pseudocompact}
By a well-known result of Gabriel, the opposite category ${}_R \PC\op$ is Grothendieck and locally finite
for any left pseudocompact ring $R$. What is not so well-known is that essentially the same proof
of~~\cite[Th\'{e}or\'{e}me~IV.3]{G} applies more generally to show that, \emph{for any (not necessarily 
left pseudocompact) topological ring $R$, the category ${}_R \PC$ is abelian and $\Cc = {}_R \PC\op$ 
is Grothendieck and locally finite.}
Gabriel's proof makes use of~\cite[Propositions~IV.10--11]{G}, both of which are purely
module-theoretic results whose proofs do not make reference to the base ring itself.
The only small adjustment that one needs to make is in identifying a cogenerator for the 
category ${}_R \PC$, as follows.
We claim that a family of cogenerators for ${}_R \PC$ is given by the discrete finite-length 
left $R$-modules. Indeed, if $f,g \colon M \to N$ are morphisms in ${}_R \PC$ with $f \neq g$, 
then there exists an element $x \in M$ and an open submodule $U \subseteq N$ such that 
the composite
\[
M \overset{f-g}{\longrightarrow} N \twoheadrightarrow N/U
\]
sends $x$ to a nonzero element. Because $N/U$ is a discrete finite-length left $R$-module,
we conclude that such modules indeed form a cogenerating set for ${}_R \PC$. 
The isomorphism classes of discrete finite-length $R$-modules form a set (as all of them are
isomorphic to $R^n/U$ for some integer $n \geq 1$ and some open submodule $U \subseteq R^n$),
so the product of a set of such representatives forms a cogenerator of ${}_R \PC$
(and a generator in the dual category ${}_R \PC\op$).
\end{remark}

An elementary observation that shall be used frequently below is that a simple left
pseudocompact (or even linearly topologized and Hausdorff) module is discrete, and 
conversely, every simple discrete left module is pseudocompact.

We will also make use of the notion of linear compactness. A left topological
module $M$ over a topological ring $R$ is said to be \emph{linearly compact} if
it is linearly topologized, Hausdorff, and every family of closed cosets in $M$
with the finite intersection property (i.e., the intersection of any finite subfamily is
nonempty) has nonempty intersection. It is known that
every linearly compact module is also complete~\cite[Theorem~28.5]{Warner},
so linearly compact modules are pro-discrete. Furthermore, it is known that 
linear compactness is preserved when passing to products and closed submodules
of linearly compact modules; see~\cite[Theorems~28.5--6]{Warner}.
Thus by Remark~\ref{rem:inverse limit} we see that linearly compact modules
are closed under (inverse) limits in the category ${}_R \TMod$.
Furthermore, since the quotient of a linearly compact module by a closed
submodule is again linearly compact~\cite[Theorem~28.3]{Warner}, we 
have the following immediate consequence of Lemma~\ref{lem:prodiscrete}, 
part of which is found in~\cite[Theorem~28.15]{Warner}.

\begin{lemma}\label{lem:linearly compact}
Let $M$ be a left topological module over a topological ring $R$. The following are equivalent:
\begin{enumerate}[label=(\alph*)]
\item $M$ is linearly compact;
\item $M$ is Hausdorff, complete, and has a neighborhood basis of $0$ consisting of open submodules
$N$ such that $M/N$ is linearly compact for the discrete topology;
\item $M$ is pro-discrete and every open submodule $N$ of $M$ is such that $M/N$ is linearly
compact for the discrete topology;
\item $M$ is an inverse limit in ${}_R \TMod$ of discrete linearly compact left $R$-modules.
\end{enumerate}
\end{lemma}

A discrete artinian module is linearly compact~\cite[Theorem~28.12]{Warner}, with the
following consequence.

\begin{corollary}\label{cor:pseudocompact is linearly compact}
Every pseudocompact left module $M$ over a topological ring $R$ is linearly compact.
\end{corollary}

\begin{proof}
By Lemma~\ref{lem:pseudocompact}, $M$ is an inverse limit of discrete modules of finite
length. These modules are linearly compact, whence Lemma~\ref{lem:linearly compact} gives
that $M$ is linearly compact.
\end{proof}

Applying this terminology directly to a topological ring $R$, we say that $R$ is
\emph{left linearly compact} if ${}_R R$ is linearly compact, \emph{right linearly compact}
if $R_R$ is linearly compact, and \emph{linearly compact} if it is both left and right
linearly compact.

\subsection{The Jacobson radical of pro-discrete rings}

We note that certain results of~\cite{G} on the Jacobson radical of pseudocompact rings 
generalize to the pro-discrete case. These will also be used to derive 
alternate proofs of results from~\cite{Z1,Z} on Jacobson radicals of prolimits of rings
and Jacobson semisimple, linearly compact rings.

We will make use of a natural topological version of the Jacobson radical for
a left pro-discrete ring $R$: we let $J_0(R)$ denote the intersection of all of
the open maximal left ideals of $R$. In the case where $R$ is (two-sided)
pro-discrete, this definition is left-right symmetric. For if $I$ is an open ideal of $R$, 
all of the left or right ideals containing $I$ are open and $R/I$ is discrete by 
Lemma~\ref{lem:open}. Thus the two radicals $J$ and $J_0$ coincide for $R/I$ 
and there is an open ideal $J_I$ of $R$ such that $J(R/I)=J_0(R/I)=J_I/I$.
Since every open maximal left or right ideal contains an open ideal $I$, we find that 
the intersection of all open maximal left ideals equals the intersection of all these $J_I$,
which also equals the intersection of all open maximal right ideals. 

The following is likely well known in some other form, but we include it for convenience:

\begin{theorem}\label{thm:WA.jq}
Let $R$ be a pro-discrete ring.
\begin{enumerate}
\item As topological rings, $R=\varprojlim R/I$ where $I$ ranges over the open ideals of $R$. The topology 
on $R$ is induced from the product topology of $\prod_{\textnormal{open }I}R/I$ via 
the usual inclusion of the inverse limit.
\item An element of $R$ is invertible in $R$ if and only if it is invertible modulo $I$ for 
every open ideal $I$.
\item $J(R)=J_0(R)$, and $J(R)$ is a closed ideal. Moreover, $x\in J(R)$ if and 
only if $x+I \in J(R/I)$ for all open ideals $I$, so that under the isomorphism of~(1), we have 
$J(R)=\varprojlim_{\textnormal{open }I}J(R/I)$.
\item $R/J(R) = \varprojlim_{\textnormal{open }I}R/J_I$, where $I\subseteq J_I\subseteq R$ is such 
that $J_I/I=J(R/I)$, and $R/J(R)$ is pro-discrete with a neighborhoods basis of $0$ consisting of 
the ideals $J_I/J(R)$.
\end{enumerate}
\end{theorem}
\begin{proof}
(1) This is a standard argument used, for example, on complete valuation rings, profinite groups, 
algebras, etc.\ (see also \cite{G}). 
It follows as in the proof of Lemma~\ref{lem:prodiscrete}, noting that the connecting homomorphisms 
in the inversely directed system of the $R/I$ for open ideals $I$ are (continuous) ring homomorphisms. 

(2) We need only prove the if clause. Assume $y+I\in R/I$ is invertible for all open $I$. Let $u_I\in R$ 
be such that $u_I+I$ is the inverse of $y+I$ (in $R/I$). If $I\subseteq L$ are open ideals and
$\pi_{I,L}:R/I\rightarrow R/L$ is the canonical surjection, then obviously $\pi_{I,L}(u_I+I)=u_L+L$, 
since the inverse of $y+L\in R/L$ is unique. 
Thus we have an element $(u_I) \in \varprojlim_{I \textnormal{ open}} R/I$, and under the identification
$R = \varprojlim R/I$ this element forms an inverse to $y$.

(3) Obviously, $J(R)\subseteq J_0(R)$. Conversely, let $x\in J_0(R)$ and let $a \in R$. 
For each open $I$,  we have $x+I \in J(R/I)$ since every maximal left (or right) ideal of the discrete 
ring $R/I$ is open; hence, $(1-ax)+ I$ is invertible in $R/I$. By~(2), $1-ax$ is invertible in $R$. 
Thus the ideal $J_0(R)$ must be contained in $J(R)$, yielding the desired equality. 
Because $J_0(R) = J(R)$ is an intersection of open, hence closed, left ideals, it is closed. 
An argument similar to the one above shows that $x \in J(R)$ if and only if $x+I \in J(R/I)$ for
all open ideals $I$. In particular, $J(R)=\varprojlim_{I \textnormal{ open}}J(R/I)$ follows.  

(4) There is a canonical map $R=\varprojlim_{I \textnormal{ open}} R/I \longrightarrow
\varprojlim_{I \textnormal{ open}} R/J_I$. The kernel of this morphism is 
$\bigcap_{I \textnormal{ open}}J_I$ and this is equal to $J(R)$ by the discussion preceding the
theorem. Consequently, the isomorphism of (4) follows. From this isomorphism we can deduce
as in the proof of Lemma~\ref{lem:prodiscrete} that $R/J(R)$ is pro-discrete with the described
neighborhood basis of $0$.
\end{proof}

We note that the above generalizes~\cite[Lemma 5]{Z1}, which describes the Jacobson radical
of a left linearly compact ring.

\subsection{An infinite Wedderburn-Artin theorem}

Given a closed submodule $N$ of a topological module $M$, we will say that $N$
has a \emph{closed complement} if there exists a closed submodule $N'$ of $M$
such that $M = N \oplus N'$. It is straightforward to verify that this is equivalent to
the condition that there exists a continuous idempotent endomorphism $e$ of $M$
such that $N = e(M)$. Note that if $N$ is an open submodule of $M$ that is a 
(``purely algebraic'') direct summand of $M$, then any complementary submodule
of $N$ is closed; for if $M = N \oplus N'$ then $M \setminus N' = 
\bigcup_{x \in N \setminus \{0\}} (x + N')$ is a union of open cosets and thus is open
in $M$.

Given an object $X$ of a Grothendieck category $\Cc$, we let $\soc(X)$ denote the
\emph{socle} of $X$ in $\Cc$, which is the colimit (``sum'') of all simple subobjects
of $X$.

\begin{lemma}\label{lem:complements}
Let $M$ be a left topological module over a topological ring $R$. The following are
equivalent:
\begin{enumerate}[label=(\alph*)]
\item $M$ is a product in ${}_R \TMod$ of simple discrete modules;
\item $M$ is pseudocompact and every closed submodule of $M$ has a closed complement;
\item $M$ is pseudocompact and every open submodule of $M$ has a complement.
\end{enumerate}
\end{lemma}

\begin{proof}
To begin, note that any module $M$ satisfying~(a) is a product of pseudocompact modules
and therefore is pseudocompact. Thus to prove the equivalence of the three conditions, we
may assume that $M$ is pseudocompact throughout. As mentioned
in Remark~\ref{rem:general pseudocompact}, the category $\Cc = {}_R \PC\op$ is locally finite
and Grothendieck. Let $M'$ denote the object in $\Cc$ ``opposite'' to $M$. Noting that products
in ${}_R \PC$ are the same as in ${}_R \TMod$, the three conditions on $M$ above translate to:
\begin{enumerate}[label=(\alph*$'$)]
\item $M'$ is a direct sum of simple objects in $\Cc$;
\item Every subobject of $M'$ is a direct summand of $M'$ in $\Cc$;
\item Every finite-length subobject of $M'$ is a direct summand of $M'$ in $\Cc$.
\end{enumerate}
The equivlence of~(a$'$) and~(b$'$) follows from well known generalizations of the usual 
module-theoretic argument to Grothendieck categories, as in~\cite[Section~V.6]{Stenstroem}.
Clearly (b$'$)$\Rightarrow$(c$'$). Finally, assume~(c$'$) holds; we will deduce~(a$'$).
Let $L$ be a finite-length subobject of $M'$. Because all subobjects of $L$ also have finite length, 
the hypothesis impiles that every subobject of $L$ is a direct summand of $M'$, including
$L$ itself, so that every subobject of $L$ is in fact a summand of $L$. Thus $L = \soc(L) \oplus L'$
for a subobject $L'$ with $\soc(L') = 0$. But $L'$ also has finite length, so that its socle being zero
means that $L' = 0$.
Thus $L = \soc(L)$ is a (direct) sum of simple objects. 
Since $\Cc$ is locally finite, $M'$ is the sum of all of its finite-length subobjects $L = \soc(L)$,
and we deduce that $M' = \soc(M')$ is a (direct) sum of simple objects.
\end{proof}

The lemma above raises an interesting question: to what extent is it possible to
characterize the structure of pro-discrete left modules over a topological ring $R$ in which
every closed submodule has a closed complement?

We now present the following infinite-dimensional version of the Wedderburn-Artin theorem
for left linearly topologized rings.

\begin{theorem}
\label{thm:WA}
Let $R$ be a topological ring. Then the following are equivalent:
\begin{enumerate}[label=(\alph*)]
\item $R$ is a product in ${}_R \TMod$ of simple discrete left modules;
\item $R$ is left pro-discrete and every closed left ideal of $R$ has a closed complement;
\item $R$ is left pro-discrete and every open left ideal of $R$ is a direct summand;
\item $R \cong \prod \End_{D_i}(V_i)$ as topological rings, where each $V_i$ is a right vector
space over a division ring $D_i$ and each $\End_{D_i}(V_i)$ is given the finite topology;
\item $R$ is left pesudocompact and $J_0(R)$ (equivalently, $J(R)$) is zero;
\item $R$ is left linearly compact and $J_0(R)$ (equivalently, $J(R)$) is zero;
\item $R$ is left pseudocompact and the abelian category ${}_R \PC$ (equivalently, ${}_R \PC\op$)
has all short exact sequences split (i.e., the category is spectral);
\item $R$ is left pseudocompact and every object in ${}_R \PC$ is a product of simple objects
(equivalently, every object in ${}_R \PC\op$ is a direct sum of simple objects, i.e., ${}_R \PC\op$
is a semisimple category).
\end{enumerate}
\end{theorem}

\begin{proof}
That (b)$\Rightarrow$(c) follows from the fact that every open left ideal is closed. Assuming that either
condition~(a) or~(c) holds, we will show that $R$ is left pseudocompact; it will then follow from 
Lemma~\ref{lem:complements} that (c)$\Rightarrow$(a)$\Rightarrow$(b). If $R$ satisfies
condition~(a), so that $R = \prod_{i \in I} S_i$ for simple discrete left modules $S_i$, then 
$R$ is a product (hence inverse limit) of discrete modules and therefore is left pro-discrete 
by Lemma~\ref{lem:prodiscrete}. Furthermore,
$R$ has a neighborhood basis of~$0$ consisting of open left ideals $L$ for which $R/L$ is a 
finite direct product of the simple $S_i$; these are the kernels of the projections $\prod_{i \in I} S_i
\twoheadrightarrow \prod_{j \in J} S_j$ for any finite subset $J \subseteq I$.
Thus $R$ is pseudocompact. 
Similarly assume~(c) holds, and let $L$ be an open left ideal of $R$. Since every left ideal 
containing $L$ is also open, and consequently a direct summand, we see that $R/L$ is 
a semisimple $R$-module. This semisimple module $R/L$ is finitely generated, and therefore
has finite length. By Lemma~\ref{lem:pseudocompact}, $R$ is left pseudocompact as claimed.

Again assume that~(a) holds. We have established that $R = \prod S_i$ is left pseudocompact 
in this case, and the intersection of the open maximal left ideals corresponding to
the kernel of each canonical projection $R \twoheadrightarrow S_i$ is zero. Thus $J_0(R) = 0$
(which is equivalent to $J(R) = 0$ by Lemma~\ref{thm:WA.jq}), establishing (a)$\Rightarrow$(e).
Also (e)$\Rightarrow$(f) by Corollary~\ref{cor:pseudocompact is linearly compact}, 
and (f)$\Rightarrow$(d) follows from the characterization of linearly compact, Jacobson
semisimple rings given in~\cite[Theorem~29.7]{Warner}. 
For (d)$\Rightarrow$(a), it suffices to show that any ring of the form $\End(V)$, for $V_D$ a vector 
space over a division ring, is a product of simple discrete modules. Fix a basis $\{v_i \mid i \in I\}$ 
of $V$ and let $E_i$ denote the projection onto $\Span(v_i)$ with kernel spanned by $\{v_j\}_{j \neq i}$. 
Then $\End(V) \cong \prod \End(V) E_i \cong \prod_{i \in I} V$ is a product of simple discrete left 
modules in ${}_R \TMod$.

Assume that~(a) holds, so that once more $R$ is left pseudocompact; we will verify~(h). 
The simple objects of
${}_R \PC$ are precisely the simple discrete modules in ${}_R \TMod$, which are all of the
form $R/U$ for an open maximal left ideal $U$ of $R$, thanks to Lemma~\ref{lem:open}. Since we
have already shown (a)$\Rightarrow$(c), we see that every such $U$ is a direct summand. 
Noting from~\cite[Corollaire~IV.1]{G} that $R$ is a projective object in ${}_R \PC$, because
$R \cong R/U \oplus U$ we find that the direct summand $R/U$ is projective in ${}_R \PC$.
So every simple object in ${}_R \PC$ is projective, and dually we have that every simple object
in the Grothendieck category $\Cc = {}_R \PC\op$ is injective. Because $\Cc$ is (locally finite 
and therefore) locally noetherian, injectives are closed under direct 
sums~\cite[Proposition~V.4.3]{Stenstroem}.  Then every simisimple object in $\Cc$, being a
direct sum of simple (hence injective) objects, must be injective.
Thus for every object $X$ in $\Cc$, the socle $\soc(X)$ is injective and therefore is 
a direct summand of $X$, so that $X = \soc(X) \oplus X'$ for some subobject $X'$ of $X$.
This $X'$ cannot have any simple subobjects as it intersects $\soc(X)$ trivially. Because $\Cc$
is locally finite, every nonzero object has a nonzero socle, from which we deduce that $X' = 0$. 
Thus $X = \soc(X)$ is a direct sum of simple objects (dually, every object of ${}_R \PC$ is a product 
of simple objects), establishing~(h). 
Conversely, if~(h) holds then $R$ is an object of ${}_R \PC$ and thus is a product of
simple objects in ${}_R \PC$. That is to say, $R$ is a product of simple discrete left 
modules, and~(a) holds. 

For (h)$\Rightarrow$(g), note that in the Grothendieck category $\Cc = {}_R \PC\op$, all objects
are semisimple, and therefore all short exact sequences split by essentially the
same argument as in the case of a module category~\cite[Section~V.6]{Stenstroem}.
Conversely, if~(g) holds and $X$ is an object of $\Cc$, then the monomorphism
$\soc(X) \hookrightarrow X$ splits, so that $X = \soc(X) \oplus X'$ for some subobject $X'$ of
$X$. Just as in the previous paragraph, local finiteness of $\Cc$ implies that $X' = 0$. Thus
$X = \soc(X)$, so that~(h) holds.
\end{proof}

The algebras above provide a suitable substitute for semisimple rings in the context of
topological algebra in which we find ourselves. Thus we introduce the following terminology.

\begin{definition}\label{def:WA}
A topological ring is called \emph{left pseudocompact semisimple} if it satisfies the 
equivalent conditions of the previous theorem. 
\end{definition}

\subsection{Some applications of the Wedderburn-Artin theorem}
We now present some special cases of Theorem~\ref{thm:WA}, some of which recover earlier
results on semisimplicity in the literature.

To begin, we note that Theorem~\ref{thm:WA} can be indeed considered as an infinite generalization
of the classical Wedderburn-Artin theorem. For if $R$ is a left artinian ring, then it is left pseudocompact
when equipped with its discrete topology. So a semisimple (left) artinian ring is left pseudocompact and
Jacobson semisimple. In the decomposition $R \cong \prod_{i\in I}\End_{D_i}(V_i)$, the topology on $R$
is discrete if and only if $I$ is finite and each $V_i$ is finite-dimensional. Also, viewing $R$ as
a product of simple discrete left $R$-modules, we again have that $R$ carries the discrete topology if
and only if this set of simple modules is finite.

In the situation of a left and right pseudocompact ring $R$, there is a more refined version of 
the structure theorem above. This will be presented after the following lemma that 
extends~\cite[Lemma 2.5]{IC}, which is the analogous statement for the case where the dual 
category ${}_R \PC$ is locally finite-dimensional over a fixed basefield. We will only need to use 
it in the pseudocompact semisimple case, but we state it in full generality and include details.

\begin{lemma}\label{lem:WA.sum}
Let $R$ be a topological ring, $(M_i)_{i\in I}$ be pseudocompact left modules and let 
$P=\prod_i M_i$.  Assume that for every simple pseudocompact left $R$-module $S$, only finitely 
many $M_i$ have $S$ as a quotient in ${}_R \PC$ (in particular, this is true if the intersection of all 
open $M\subseteq P$ for which  $P/M\cong S$ is of finite colength in $P$). Then the coproduct
of the family $(M_i)_i$ in ${}_R \PC$ is $P$, with the obvious canonical maps.
\end{lemma}

\begin{proof}
We will show that $P$ with the canonical morphisms $\sigma_i \colon M_i \hookrightarrow P$ 
satisfies the universal property of the coproduct in ${}_R \PC$. 
Let $\Sigma=\bigoplus_i M_i$ denote the usual direct sum of the family $(M_i)_i$, forming the 
coproduct in the category of left $R$-modules. Regard $\Sigma \subseteq P$ as a submodule 
of the product.

Let $N$ be a pseudocompact module with continuous maps $f_i:M_i\rightarrow N$, and let 
$f \colon \Sigma \rightarrow N$ be the unique canonical $R$-module map with $f=f_i$ on $M_i$ 
(i.e., $f \sigma_i = f_i$). We will first show that $f$ is continuous (with respect to the topology on
$\Sigma$ inherited from $P$), then that $f$ extends to a continuous homomorphism
$P \to N$ satisfying the desired universal property. 

Because $f$ is linear, to show that $f$ is continuous it suffices to show that, for any open submodule
$H$ of $N$, $f^{-1}(H)$ is open in $M$.
We claim that  $M_i\subset f^{-1}(H)$
for all but finitely many $i$. Indeed, let $\Ss$ be the set of simple modules that occur in the 
Jordan-H${\rm\ddot{o}}$lder series of $N/H$. These simple modules are discrete and therefore
belong to ${}_R \PC$. 
If $M_i$ is such that $M_i\nsubseteq f^{-1}(H)$ then $f^{-1}(H)\cap M_i = f_i^{-1}(H)$ is 
open (by continuity of $f_i$) and proper in $M_i$. Now the composite map 
\[
M_i \overset{f_i}{\longrightarrow} N \twoheadrightarrow N/H
\] 
has image isomorphic to $M_i/f_i^{-1}(H) = M_i/(f^{-1}(H)\cap M_i)$, making the latter a nonzero 
module that embeds in $N/H$. 
Being nonzero, this finite-length module has some simple quotients that lie in the family
$\Ss$. We conclude that $M_i$ has some simple module $S\in\Ss$ as a quotient. 
Because $\Ss$ is finite, the hypothesis on the $M_i$ ensures that we can have at most finitely many 
$M_i \nsubseteq f^{-1}(M)$, as claimed.
Thus $f^{-1}(H)$ contains $\bigoplus_{i\notin F}M_i$ for some finite subset $F\subseteq I$, from
which it readily follows that  $f^{-1}(H)\supseteq \Sigma\cap \left(\prod_{i\in F}f_i^{-1}(H) 
\times \prod_{i\notin F}M_i\right)$. The latter submodule is open in $\Sigma$ (being the restriction
of an open submodule of $P$), so that $f^{-1}(H)$ is open by Lemma~\ref{lem:open}. 
Hence $f$ is continuous.

Note that every element $x \in P$ of the product is the limit of the Cauchy net $(x_J)$ in $\Sigma$ 
indexed by the finite subsets $J \subseteq I$, where $x_J$ has the same entries as $x$ at each index
in $J$ and all other entries zero. (In particular, $\Sigma$ is dense in $P$.) Because $f \colon \Sigma \to N$
is linear and continuous, each such Cauchy net $(x_J)$ is mapped to a Cauchy net $(f(x_J))$ in $N$.
Setting $\overline{f}(x) = \lim_J f(x_J)$ to be the limit of this net, one may verify that
$\overline{f} \colon P \to N$ is $R$-linear (using the fact that the assignment $x \mapsto (x_J)$ is 
$R$-linear) and continuous (see also~\cite[Theorem~7.19]{Warner}). 
This map satisfies $\overline{f} \circ \sigma_i = f_i$ for all $i$. Further, it is the unique continuous
homomorphism satisfying this condition: any such map restricts to $f$ on the dense subset
$\Sigma \subseteq P$, and a continuous function $P \to N$ is uniquely determined by its restriction to 
a dense subset thanks to the Hausdorff property of $N$.
\end{proof}

\begin{theorem}
\label{thm:WA.1}
Let $R$ be a topological ring. The following assertions are equivalent:
\begin{enumerate}[label=(\alph*)]
\item $R = \prod_{i \in I} S_i^{\, n_i}$ in ${}_R \TMod$ for some pairwise non-isomorphic
discrete simple left $R$-modules $S_i$ and integers $n_i \geq 1$;
\item $R$ is pseudocompact and ${}_R R$ is a coproduct of simple objects in the category ${}_R \PC$; 
\item $R \cong \prod_i \M_{n_i}(D_i)$ as topologial rings for some division rings $D_i$ 
and integers $n_i \geq 1$, where each matrix ring is given the discrete topology;
\item $R$ is pro-discrete and every open (respectively, closed) left ideal has a (closed)
complement;
\item $R$ is left pseudocompact semisimple and right linearly topologized;
\item $R$ is pseudocompact and $J(R)$ (equivalently, $J_0(R)$) is zero;
\item $R$ is linearly compact and $J(R)$ (equivalently, $J_0(R)$) is zero;
\item The left-right symmetric statements of (a), (b), (d), and (e).
\end{enumerate}
\end{theorem}

\begin{proof}
First note that if~(c) holds, then $R$ is a product of pseudocompact semisimple (hence pro-discrete 
and linearly compact) rings. Thus~(c) implies all of~(d), (e), (f), and~(g). Conversely, if any one of
the conditions~(d)--(g) holds, then $R$ is right linearly topologized and by 
Theorem~\ref{thm:WA} we have $R \cong \prod_{i \in I} \End_{D_i}(V_i)$ as topological rings for 
some right vector spaces $V_i$ over divsion rings $D_i$. In particular, each of the factors 
$\End_{D_i}(V_i)$ in the product is right linearly topologized.
It follows from Remark~\ref{rem:WA.not} that each $V_i$ has finite dimension (say 
$n_i = \dim_{D_i}(V_i)$), in which case the finite topology on $\End_{D_i}(V_i) \cong \M_{n_i}(D_i)$ 
is discrete.
This establishes the equivalence of~(c)--(g).

Next, assume that~(c) holds. Let $S_i = D_i^{n_i}$, which is a simple discrete right $R$-module
via the projection $\pi_i \colon R \twoheadrightarrow \M_{n_i}(D_i)$. It is straightforward to show 
that the annihilator of ${}_R S_i$ is the kernel of the projection $\pi_i$. As $\ker(\pi_i) \neq \ker(\pi_j)$ 
for distinct $i,j \in J$, we deduce that the $S_i$ are pairwise non-isomorphic.
Then because each $\M_{n_i}(D_i) \cong S_i^{n_i}$ as left $R$-modules, and because the isomorphism in~(c) is also an isomorphism in ${}_R \TMod$, we find that (c)$\Rightarrow$(a). 

Next let $R$ satisfy~(a); we will deduce~(b). Certainly such $R$ is left 
pseudocompact semisimple, and in ${}_R \PC$ we have $R = \coprod S_i^{\, n_i}$ is a 
coproduct of simple objects according to Lemma~\ref{lem:WA.sum}. 
To see that $R$ is right pseudocompact, let $J_i \subseteq R$ denote the kernel of the canonical
projection $R \twoheadrightarrow S_i^{\,n_i}$. 
Each $J_i$ is an open left ideal of finite colength, and the $J_i$ form a basis of open neighborhoods 
of zero thanks to the structure of $R$.
It suffices to show that each $J_i$ is also a right ideal, for then $R/J_i$ is a semisimple artinian
ring. 
To this end, fix $r \in R$ and let $\rho \colon R \to R$ be the continuous endomorphism defined 
by $\rho(x) = xr$. Note that each $S_j^{n_j}$ is the isotypic component of $R$ corresponding to $S_j$
(i.e., the sum of all ${}_R \PC$-subobjects of $R$ that are isomorphic to $S_j$).  
It is clear from the construction of $J_i$ that $J_i = \coprod_{j \neq i} S_j^{\,n_j}$ is the coproduct
in ${}_R \PC$ of the isotypic components of $R$ corresponding to the simple objects $S_j$ with $j \neq i$. 
Since each isotypic component $S_j^{\, n_j}$ is invariant under $\rho$, the same is true of $J_i$.
It follows that $J_i r = \rho(J_i) \subseteq J_i$, making $J_i$ a right ideal as desired. 

For (b)$\Rightarrow$(e), suppose $R$ is pseudocompact and $R = \coprod L_t$ is a coproduct
of simple objects in ${}_R \PC$. Certainly $R$ is right linearly topologized. To see that
it is Jacobson semisimple,  let $M_t$ denote the open maximal left ideal that is the kernel 
of the projection $R \to L_t$.
It follows from $R = \coprod L_t$ that $\bigcap M_t = 0$, so that $J(R) = 0$. 

This establishes the equivalence of conditions (a)--(g).
Finally, condition~(h) is equivalent to the rest because properties~(f) and~(g) are left-right symmetric.
\end{proof}

We note that the equivalence (b)$\Leftrightarrow$(g) above recovers~\cite[Theorem 1]{Z}, which in
turn generalized \cite[Theorem~16]{K} using the methods of~\cite{Z1}.

We will also briefly demonstrate that the above results can be particularized to yield the characterization of cosemisimple 
coalgebras. For the definitions of coalgebras and their comodules we refer the reader to \cite{DNR, R}.

The following is a more refined notion of pseudocompactness for $\K$-algebras, which is satisfied both
by the diagonalizable algebras that we consider in Section~\ref{sec:diagonal}, as well as dual algebras
of coalgebras. 

\begin{definition}\label{def:K-pseudocompact}
Given a topological algebra $A$ over a field $\K$, we say that a left topological $A$-module $M$ is 
\emph{$\K$-pseudocompact} if it is pseudocompact and its open submodules have finite 
$\K$-codimension. We say  that $A$ is \emph{left $\K$-pseudocompact} (or 
\emph{left pseudocompact as a $\K$-algebra}) 
if ${}_A A$ is $\K$-pseudocompact. (Right $\K$-pseudocompact algebras are defined similarly.) We say 
that $A$ is \emph{$\K$-pseudocompact} if it pseudocompact and every open ideal of $A$ has finite
$\K$-codimension.
\end{definition}

Because a module of finite length over an algebra is finite-dimensional if and only if all of its simple
composition factors are finite-dimensional, one can see that a topological algebra $A$ is left
$\K$-pseudocompact if and only if it is a left pseudocompact topological ring and every simple discrete
left $A$-module has finite $\K$-dimension. 

Note that, as in Remark~\ref{rem:left and right pseudocompact}, a topological $\K$-algebra is 
$\K$-pseudocompact if and only if it is both left and right $\K$-pseudocompact.
If $A$ is a left $\K$-pseudocompact algebra, then every left pseudocompact $A$-module is 
$\K$-pseudocompact, so that ${}_A \PC$ coincides with the category of $\K$-pseudocompact 
left $A$-modules (with continuous module homomorphisms).

The Wedderburn-Artin theorems proved above restrict to $\K$-pseudocompact algebras in the following way.

\begin{proposition}\label{prop:WA fin dim}
Let $A$ be a topological $\K$-algebra. Then the following are equivalent:
\begin{enumerate}[label=(\alph*)]
\item $A$ is left $\K$-pseudocompact and Jacobson semisimple;
\item $A$ is $\K$-pseudocompact and Jacobson semisimple;
\item $A \cong \prod_{i \in I} \M_n(D_i)$ as topological algebras, where each $D_i$ is
a finite-dimensional division algebra over $\K$ and each $\M_n(D_i)$ is given the discrete 
topology.
\end{enumerate}
\end{proposition}

\begin{proof}
Clearly we have (c)$\Rightarrow$(b)$\Rightarrow$(a). Finally, assuming~(a) holds, it follows
from Theorem~\ref{thm:WA} that $A \cong \prod_{i \in I} \End_{D_i} (V_i)$ as topological
algebras for some division $\K$-algebras $D_i$ and right $D_i$-vector spaces $V_i$. Note
that each $V_i$ is a discrete simple left $A$-module via the projection
$A \twoheadrightarrow \End_{D_i} (V_i)$. If there exists $i$ such that either $D_i$ is infinite-dimensional 
over $\K$ or $V_i$ is infinite-dimensional over $D_i$, then the corresponding
discrete simple left $A$-module $V_i$ is infinite-dimensional over $\K$, contradicting left 
$\K$-pseudocompactness of $A$. Thus all of the $D_i$ are finite-dimensional $\K$-algebras
and all $V_i$ are finite-dimensional $D_i$-vector spaces. Now~(c) readily follows.
\end{proof}

As a corollary, using the duality between pseudocompact algebras and coalgebras, one can also 
deduce known results in the basic theory of coalgebras, namely, the characterization of cosemisimple 
coalgebras; see \cite[Theorem 3.1.5]{DNR}.

The category of $\K$-pseudocompact algebras (with continuous algebra homomorphisms) is in 
duality with the category of coalgebras over $\K$; we refer readers to~\cite{Simson} for details, but
we sketch the ideas here.
Given a coalgebra $C$, its $\K$-dual algebra $C^*$ is pseudocompact with open ideals being those 
of the following form, where $H$ is a finite-dimensional subcoalgebra of $C$:
\[
H^\perp = \{\phi \in C^* \mid \phi(H) = 0\}.
\]
Conversely, if $A$ is a $\K$-pseudocompact algebra, then one can form its ``finite dual'' coalgebra
\[
A^\circ = \varinjlim (A/I)^*,
\]
where $I$ ranges over the open ideals of $A$. The assignments $C \mapsto C^*$ and $A \mapsto A^\circ$
are functors that provide a duality between the categories of $\K$-coalgebras (with coalgebra morphisms)
and $\K$-pseudocompact algebras (with continuous homomorphisms)~\cite[Theorem~3.6]{Simson}.

Similarly, if $M$ is a left comodule over a coalgebra $C$, then its dual $M^*$ naturally carries the
structure of a left $C^*$-module. When $M^*$ is equipped with the topology whose open submodules are
of the form
\[
N^\perp = \{\phi \in M^* \mid \phi(N) = 0\}
\]
for finite-dimensional subcomodules $N \subseteq M$, it becomes a $\K$-pseudocompact left
$C^*$-module~\cite[Corollary~2.2.13]{DNR}. In fact, this provides an duality between the category
of left $C$-comodules and the category ${}_A \PC$ for the $\K$-pseudocompact algebra $A = C^*$;
see~\cite[Theorem~4.3]{Simson}.

\begin{corollary}
Let $C$ be a coalgebra over a field $\K$. The following assertions are equivalent:
\begin{enumerate}[label=(\alph*)]
\item $C$ is a direct sum of simple left (equivalently, right) comodules, i.e., it is cosemisimple;
\item Every left (equivalently, right) $C$-comodule is semisimple;
\item Every short exact sequence of left (equivalently, right) $C$-comodules splits;
\item $C^* \cong \prod_{i}\M_{n_i}(D_i)$ as topological algebras, for some finite-dimensional 
division algebras $D_i$;
\item $C$ is a direct sum of coalgebras of the form $\M_n(D)^*$ (dual to the algebras $\M_n(D)$), 
with $D$ a finite-dimensional division algebra over $\K$.
\end{enumerate}
\end{corollary}
\begin{proof}
The dual algebra $A = C^*$ is $\K$-pseudocompact, and under the duality between $C$-comodules
and pseudocompact $C^*$-modules, conditions (a)--(c) above translate to the following:
\begin{enumerate}[label=(\alph*$'$)]
\item $A$ is a product of simple pseudocompact left (equivalently, right) modules;
\item Every object in ${}_A \PC$ (equivalently, $\PC_A$) is a product of simple objects;
\item Every short exact sequence in ${}_A \PC$ (equivalently, $\PC_A$) splits.
\end{enumerate}
Thus the equivalence of~(a)--(d) follows from Theorem~\ref{thm:WA} and 
Proposition~\ref{prop:WA fin dim}.
Also,~(e) is the dual of~(d), with $\M_n(D)^*$ being the finite-dimensional coalgebra dual 
to the matrix algebra $\M_n(D)$.
\end{proof}

\begin{remark}
We note that if $A$ is a $\K$-pseudocompact algebra, asking that every pseudocompact left $A$-module be semisimple as an object in ${}_A \PC$ is not equivalent to $A$ being semisimple in the above sense. 
This fails for set-theoretic reasons that cannot be avoided. Translated into the dual category of comodules over $C = A^*$, the statement that the category ${}_A \PC$ is semisimple would mean that every $C$-comodule is a direct product of simple comodules. However, taking $S$ to be a simple comodule, the comodule $L = S^{(\aleph_0)}$ is not a product of simple comodules. If it were so, it is easy to see that we
would have $L \cong S^{\Omega}$ for some infinite cardinal $\Omega$. 
Hence, its dimension over $\K$ would be at least $2^{\aleph_0}$ since $\Omega \geq \aleph_0$. But
$\dim_\K(S^{(\aleph_0)})=\aleph_0$ since $\dim_\K(S)<\infty$, and this is a contradiction.
\end{remark}

\section{Diagonalizable algebras of operators} \label{sec:diagonal}

This final section begins with a detailed study of the structure of diagonalizable
algebras of operators. We then prove our major theorem characterizing diagonalizable
subalgebras in terms of their structure as topological algebras. At the end of the section, we 
investigate the closure of the set of diagonalizable operators. 

\subsection{Structure of diagonalizable subalgebras}

Let $V$ be a finite-dimensional vector space over a field $\K$,
and let $T \in \End(V)$. Suppose that $T$ is diagonalizable, so that its minimal polynomial
$\mu(x) = (x- \lambda_1) \cdots (x - \lambda_d)$ is a product of distinct linear factors. 
The subalgebra $\K[T] \subseteq \End(V)$ is then isomorphic to $\K[x]/(\mu(x)) \cong \K^d$
(the latter isomorphism following readily from the Chinese Remainder Theorem). In terms of
algebraic geometry, the induced surjective homomorphism $\K[x] \twoheadrightarrow \K[x]/(\mu(x))
\cong \K[T]$ is dual to the inclusion of algebraic varieties $\{\lambda_1, \dots, \lambda_d\}
\hookrightarrow \mathbb{A}^1_\K$ of the $d$ distinct eigenvalues of $T$ into the affine line
over $\K$. Thus the subalgebra $\K[T]$ can be viewed in algebro-geometric terms as being
isomorphic to the algebra of $\K$-valued functions on a set of $d$ points. In particular, the
prime spectrum of $\K[T]$ is a discrete space consisting of $d$ distinct points, so that 
$\Spec(\K[T]) \cong \{\lambda_1, \dots, \lambda_d\}$.

This is certainly not the typical picture of spectral theory presented in linear algebra textbooks,
but we will show in this section that such a view of diagonalizable algebras of operators as algebras 
of functions generalizes to the infinite-dimensional case. In particular, every diagonalizable subalgebra 
of $\End(V)$ is isomorphic to an algebra $\K^X$ of $\K$-valued functions on a (possibly infinite) set 
$X$, even as topological algebras. We work at the level of categories and functors, to present
a kind of duality between functions and underlying sets as in the brief illustration above in terms
of algebraic geometry. We refer readers to~\cite[Chapter~I]{MacLane} for the basic category theory 
required here. In this subsection, we shall follow standard category-theoretic practice and view a
\emph{contravariant} (``arrow-reversing'') functor $F \colon \Cc_1 \to \Cc_2$ as a covariant functor 
$F \colon \Cc_1\op \to \Cc_2$ out of the opposite category.

The ``underlying set'' of one of these function algebras can be viewed as a kind of prime spectrum,
but suitably modified to fit with the context of topological algebra in which we work.
Given a commutative topological ring $A$, we let $\Spec_0(A)$ denote the set of open
prime ideals of $A$. 
Let $\TRing$ denote the category of topological rings with continuous ring homomorphisms,
and let $\cTRing$ denote the full subcategory of commutative topological rings. If 
$f \colon A \to B$ is a morphism in $\cTRing$ and $\p \in \Spec_0(B$), then it is clear that
$f^{-1}(\p) \in \Spec_0(A)$. In this way, the assignment $A \mapsto \Spec_0(A)$ forms a 
functor $\Spec_0 \colon \cTRing\op \to \Set$. We will show in Theorem~\ref{thm:duality} below 
that this becomes an equivalence when restricted to a suitable full subcategory.
We begin with a more detailed account of the Jacobson radical and semisimplicity for commutative
pseudocompact rings.

In the case where $A$ above is pseudocompact, if $\m \in \Spec_0(A)$ then $A/\m$ is an artinian 
integral domain and therefore is a field, so that $\m$ is maximal. Thus when $A$ is pseudocompact,
$\Spec_0(A)$ is the set of open \emph{maximal} ideals of $A$, and the assignment sending a
commutative psuedocompact topological ring to its set of open maximal ideals is functorial (coinciding 
with $\Spec_0$).

Recall that every element of the Jacobson radical of a commutative artinian ring is nilpotent;
we shall show that the radical of a commutative pseudocompact ring satisfies a similar condition.
An element $x$ of a topological ring is \emph{topologically nilpotent} if the sequence
$(x^n)_{n=1}^\infty$ converges to zero. 
Given a commutative topological ring $A$, let $N_0(A)$ denote the set of all topologically
nilpotent elements of $A$. 

\begin{lemma}
Let $A$ be a commutative linearly topologized ring.
\begin{enumerate}
\item $N_0(A) = \bigcap \Spec_0(A)$; consequently, $N_0(A)$ is a closed ideal of $A$ contained 
in $J(A)$.
\item If $A$ is pseudocompact, then $N_0(A) = J(A)$.
\end{enumerate}
\end{lemma}

\begin{proof}
(1) Set $N = N_0(A)$. 
Let $x \in N$ and $\p \in \Spec_0(A)$. Since $\p$ is an open neighborhood of $0$ and
$\lim_{n \to \infty} x^n  = 0$, there is an integer $n \geq 1$ such that $x^n \in \p$. Since
$\p$ is prime, this means that $x \in \p$. We conclude that $N \subseteq \bigcap \Spec_0(A)$.
Now suppose that $x \in A \setminus N$; then there is an open neighborhood $I \subseteq A$ 
of~$0$, which we may assume is an ideal, such that set $S = \{x^n \mid n \geq 1\}$ is disjoint
from $I$. By a familiar application of Zorn's Lemma, there is an ideal $\p$ of $A$ with 
$I \subseteq \p$ (making $\p$ open) that is maximal with respect to $\p \cap S = \varnothing$.
Since $S$ is a multiplicatively closed set, this $\p$ is prime by a well known argument from
commutative algebra. Thus $\p \in \Spec_0(A)$ with $x \notin \p$, so that 
$x \notin \bigcap \Spec_0(A)$. Now we see that $N_0(A)$ is closed as it is an intersection of
open, hence closed, ideals.

(2) If $A$ is pseudocompact, then $\Spec_0(A)$ is the set of open maximal ideals of $A$, as
noted above. Thus $N_0(A) = \bigcap \Spec_0(A) = J_0(A) = J(A)$ thanks to part~(1) above
and Theorem~\ref{thm:WA.jq}(3).
\end{proof}

We say that a topological ring is \emph{topologically reduced} if its only topologically nilpotent
element is zero.

\begin{lemma}\label{lem:commutative semisimple}
Let $A$ be a commutative topological ring.  The following are equivalent:
\begin{enumerate}[label=(\alph*)]
\item $A$ is pseudocompact semisimple; 
\item $A$ is pseudocompact and topologically reduced;
\item The open maximal ideals of $A$ intersect to zero and form a neighborhood subbasis
for $0$, and $A$ is complete;
\item As a topological ring, $A$ is a product of discrete fields;
\item The canonical map $A \to \prod_{\m \in \Spec_0(A)} A/\mathfrak{m}$ is an isomorphism
of topological rings.
\end{enumerate}
\end{lemma}

\begin{proof}
It is clear that (e)$\Rightarrow$(d)$\Rightarrow$(c)$\Rightarrow$(a), and (a)$\Leftrightarrow$(b)
follows from the previous lemma. Theorem~\ref{thm:WA}
gives (a)$\Rightarrow$(d), as the endomorphism ring of a vector space over a division ring
is commutative if and only if the division ring is a field and the vector space has dimension~$1$.
Finally, assume~(d) holds; we will show~(e). Suppose $A \cong \prod_{i \in I} \K_i$ for some (discrete)
fields $\K_i$, and for any $j \in I$ let $\m_j$ denote the kernel of the projection 
$A \cong \prod \K_i \twoheadrightarrow \K_j$. Then every open ideal of $A$ must contain a finite
intersection of the $\m_i$. In particular, any open prime ideal of $A$ contains a finite intersection
of these $\m_i$ and therefore (contains and) is equal to one of the $\m_i$. So $\Spec_0(A) = 
\{\m_i \mid i \in I\}$, from which~(e) readily follows. 
\end{proof}

Given a (discrete) field $\K$, topological $\K$-algebras with continuous algebra homomorphisms form 
a category $\TAlg_\K$. We denote the hom-sets of this category simply by $\Hom(A,B)$.

We next concern ourselves with a study of topological $\K$-algebras that are isomorphic to $\K^X$.
For any set $X$, the natural identification 
\[
\K^X = \Set(X,\K)
\]
can be viewed as an identification of topological $\K$-algebras, where the topology on $\K^X$ is 
the product topology and the topology of $\Set(X,\K)$ is the topology of pointwise 
convergence as described in Section~\ref{sec:topology} (in both cases, $\K$ is endowed with 
the discrete topology).
This makes it clear that the assignment $X \mapsto \K^X$ is the same as (the object part of) the
representable functor $\Set(-,\K) \colon \Set\op \to \TAlg_\K$ from the category 
of sets to the category of topological $\K$-algebras. Thus we shall interchange notation at the
level of functors: $\K^- = \Set(-,\K)$. 

Given a set $X$ and $x \in X$, let $\m_x \subseteq \K^X$ denote the open maximal ideal
consisting of those functions that vanish at $x$, and let $\ev_x \colon \K^X \to \K$ denote
the continuous homomorphism given by evaluating at $x$, so that $\ev_x(f) = f(x)$ for $f \in \K^X$. 
(These are the same as the canonical projections when $\K^X$ is viewed as a product of sets,
but reinterpreted in the language of functions in order to evoke appropriate imagery from
algebraic geometry.)
Clearly each $\ker(\ev_x) = \m_x$. Further, the topology on $\K^X$ has a neighborhood basis
of open ideals given by the finite intersections of the $\m_x$. In particular, each open ideal contains 
a finite product of the $\m_x$, so every open \emph{prime} ideal must be equal to some $\m_x$. 
Thus $\Spec_0(\K^X) = \{\m_x \mid x \in X\}$, and we obtain a canonical bijection 
$\Spec_0(\K^X) \overset{\sim}{\longrightarrow} \Hom(\K^X,\K)$ given by  $\m_x \mapsto \ev_x$.

\begin{lemma}\label{lem:function algebra}
Given a field $\K$ and a topological $\K$-algebra $A$, the following are equivalent:
\begin{enumerate}[label=(\alph*)]
\item $A$ is pseudocompact semisimple, and all of its open maximal ideals have $\K$-codimension
equal to~1 in $A$;
\item $A \cong \K^X$ as topological algebras for some set $X$;
\item The natural map $A \to \K^{\Hom(A,\K)}$ given by $f \mapsto (\psi(f))_{\psi}$ is an isomorphism
of topological algebras;
\end{enumerate}
and for such an algebra $A$ the canonical map $\Hom(A,\K) \to \Spec_0(A)$ given by 
$\psi \mapsto \ker(\psi)$ is a bijection.

If $\K$ is algebraically closed, then the above are further equivalent to:
\begin{enumerate}[label=(\alph*), resume]
\item $A$ is $\K$-pseudocompact and Jacobson semisimple;
\item $A$ is $\K$-pseudocompact and topologically reduced.
\end{enumerate}
\end{lemma}

\begin{proof}
We have (c)$\Rightarrow$(b), and (b)$\Rightarrow$(a) follows by combining
Lemma~\ref{lem:commutative semisimple} with the facts that $\Spec_0(\K^X) = \{\m_x \mid x \in X\}$
and $\K^X = (\K \cdot 1) \oplus \m_x$. Now assuming~(a), we shall
derive~(c). Under this hypothesis, for each $\m \in \Spec_0(A)$ we have $A = (\K \cdot 1) \oplus \m$,
yielding an isomorphism $A/\m \cong \K$ via $\lambda + \m \mapsto \lambda$. 
So every such $\m$ is the kernel of some $\psi_\m \in \Hom(A,\K)$
(the composite $A \twoheadrightarrow A/\m \cong \K)$, and this $\psi_\m$ is unique thanks
to the decomposition $A = (\K \cdot 1) \oplus \m$. Conversely, every $\psi \in \Hom(A,\K)$
is of the form $\psi = \psi_\m$ for $\m = \ker(\psi) \in \Spec_0(A)$.
Now let $\alpha \colon A \to \prod_{\m \in \Spec_0} A/\m$ be the canonical isomorphism
provided by Lemma~\ref{lem:commutative semisimple}, and consider the isomorphism
$\beta \colon \prod_{\m \in \Spec_0(A)} A/\m \to \prod_{\psi \in \Hom(A,\K)} \K$ defined by
$(\lambda_\m + \m)_\m \to (\lambda_\m)_{\psi_\m}$.
Then the composite isomorphism in $\TAlg_\K$
\[
\xymatrix{
A \ar[r]^-\alpha & \prod_{\m \in \Spec_0(A)} A/\m \ar[r]^-\beta & \prod_{\psi \in \Hom(A,\K)} \K
}
\]
coincides with the natural map described in~(c). Thus (a)$\Rightarrow$(c).

Finally, assume that $\K$ is algebraically closed. Note that (d)$\Leftrightarrow$(e) thanks to 
Lemma~\ref{lem:commutative semisimple}. Certainly (a)$\Rightarrow$(d). Conversely, suppose
that~(d) holds. To verify~(a), let $\m \in \Spec_0(A)$. Then $\m$ has finite codimension, making 
$A/\m$ a finite field extension of $\K$. Because $\K$ is algebraically closed, we must have
$A/\m \cong \K$ as $\K$-algebras, meaning that $\m$ has codimension~1 as desired.
\end{proof}

\begin{definition}
We will refer to a topological algebra over a field $\K$ satisfying the equivalent conditions above 
as a  \emph{function algebra over $\K$}, or a \emph{$\K$-function algebra}.
\end{definition}

The importance of function algebras when studying diagonalizable algebras of operators is due to
the following two facts.

\begin{proposition}\label{prop:diagonal structure}
Let $V$ be a $\K$-vector space and let $\B$ be a basis for $V$. The commutative subalgebra $A \subseteq \End(V)$ 
consisting of $\B$-diagonalizable operators is a function algebra, isomorphic as a topological $\K$-algebra to $\K^\B$.
Furthermore, $A$ is a maximal commutative subalgebra.
\end{proposition}

\begin{proof}
For each element $f \in \K^\B$ considered as a function $f \colon \B \to \K$, there is a 
corresponding operator $T_f$ on $V$ defined by $T_f(b) = f(b) \cdot b$. It is quite clear that 
this defines an algebra isomorphism $\phi \colon \K^\B \to A$ by $\phi(f) = T_f$.
To see that this is a homeomorphism, we simply note that the basic open neighborhoods of zero in 
$\K^\B$ of the form $\m_{b_1} \cap \cdots \cap \m_{b_n}$ for $b_1, \dots, b_n \in \B$ correspond
under $\phi$ to the basic open neighborhoods of zero in $A$ of the form 
$\{T \in A \mid T(b_1) = \cdots = T(b_n) = 0\}$.

To see that $A$ is a maximal commutative subalgebra of $\End(V)$, fix an element $T \in A'$ of
the commutant of $A$ in $\End(V)$; we will show that $T \in A$. For each $b \in \B$, let $E_b \in A$ 
denote the projection of $V$ onto $\K b$ with kernel spanned 
by $\B \setminus \{b\}$. Then $T$ centralizes the $E_b$, from which we deduce that the 
1-dimensional subspaces $\K b \subseteq V$ are invariant under $T$ for all $b \in \B$. It follows 
that $T$ is $\B$-diagonalizable, which is to say that $T \in A$, as desired.
\end{proof}

\begin{lemma}\label{lem:closed function subalg}
Every closed $\K$-subalgebra of a $\K$-function algebra is again a $\K$-function algebra.
\end{lemma}

\begin{proof}
Let $B$ be a $\K$-function algebra and let $A \subseteq B$ be a closed subalgebra.
We will show that condition~(a) of Lemma~\ref{lem:function algebra} passes from $B$ to 
$A$. Because $A$ is closed in $B$ and $B$ is complete, $A$ is also complete. Also, $A$ is 
Hausdorff because $B$ is. Now let $\{\m_i\}$ denote the open maximal ideals of $B$ and set 
$\m_i' = \m_i \cap A$. These are open ideals of $A$, which still intersect to zero (since 
$\bigcap \m_i = 0$) and form a neighborhood subbasis of $0$ in $A$, making $A$ Jacobson
semisimple. 
Furthermore, as each $\m_i$ has codimension~1 in $B$, the same remains true of the $\m'_i$ 
in $A$. In particular, each $\m_i'$ is an open maximal ideal of $A$. 
To see that $A$ is pseudocompact, let $\U$ denote the neighborhood basis of $0$ in $A$
consisting of the intersections of finite subfamilies of $\{\m'_i\}$. By the Chinese Remainder
Theorem, for each $U \in \U$ the ring $A/U$ is a finite direct product of fields (hence artinian).
Thus $A$ is pseudocompact by Lemma~\ref{lem:pseudocompact}(a). As we have shown that $A$
is Jacobson semisimple with every open maximal ideal of codimension~1, we find that $A$ is
a function algebra.
\end{proof}

Now let $\Func(\K)$ denote the category of $\K$-function algebras with continuous $\K$-algebra
homomorphisms.  The discrete algebra $\K$ is an object of $\Func(\K)$, and the kernel of each
homomorphism $A \to \K$ in $\Func(\K)$ is an open maximal ideal of $A$, which is to say an
element of $\Spec_0(A)$. Conversely, if $\m \in \Spec_0(A)$ for a $\K$-function algebra $A$,
then $A/\m \cong \K$ as topological $\K$-algebras, giving a continuous $\K$-algebra morphism
$A \to \K$ with kernel $\m$. Thus we have a natural isomorphism between the functors
\begin{align*}
\Spec_0 &\colon \Func(K)\op \to \Set \quad \mbox{and} \\
\Hom(-, \K) &\colon \Func(K)\op \to \Set 
\end{align*}
given (in one direction) by sending $f \in \Hom(A,\K)$ to $\ker(f) \in \Spec_0(A)$. (Note that
this generalizes the case where $A = \K^X$ discussed before Lemma~\ref{lem:function algebra}.)

The formalities developed above allow us to describe $\K$-function algebras and the morphisms
relating them in the following precise way.

\begin{theorem}\label{thm:duality}
Let $\Func(\K)$ denote the category of $\K$-function algebras with continuous $\K$-algebra
homomorphisms. Then the representable functors
\begin{align*}
\Spec_0 \cong \Hom(-,\K) &\colon \Func(\K)\op \to \Set \\
\K^{-} = \Set(-,\K) &\colon \Set\op \to \Func(K)
\end{align*}
form a contravariant equivalence between $\Func(\K)$ and $\Set$.
\end{theorem}

\begin{proof}
Given a set $X$, as previously described we have a bijection $\eta_X \colon X 
\overset{\sim}{\longrightarrow} \Hom(\K^X,\K)$ given by $x \mapsto \ev_x$ for $x \in X$, which we
argue is natural in $X$. Fixing a morphism $\phi \colon X \to Y$ in $\Set$, we have an induced
morphism $\phi^* = \K^\phi \colon \K^Y \to \K^X$ in $\TAlg_\K$, which precomposes each
function in $\K^X$ with $\phi$. 
To describe the function $\Hom(\phi^*,\K) \colon \Hom(\K^X,\K) \to \Hom(\K^Y,\K)$, fix $x \in X$ 
and the corresponding evaluation map $\ev_x \in \Hom(\K^X,K)$. Then
$\Hom(\phi^*,\K)(\ev_x) = \ev_x \circ \phi^* \colon \K^X \to \K$. Given $f \in \K^X$, we have
\[
\ev_x \circ \phi^*(f) = \ev_x(\phi^*(f)) = \ev_x(f \circ \phi) = f(\phi(x)) = \ev_{\phi(x)}(f).
\]
So $\Hom(\phi^*,\K)(\ev_x) = \ev_{\phi(x)}$, verifying that the diagram
\[
\xymatrix{
X \ar[rr]^-{\eta_X} \ar[d]_{\phi} & & \Hom(\K^X,\K) \ar[d]^{\Hom(\K^\phi,\K)} \\
Y \ar[rr]^-{\eta_Y} & & \Hom(\K^Y,\K)
}
\]
commutes. Thus the $\eta_X$ form components of a natural isomorphism 
$\eta \colon \mathbf{1}_\Set \to \Hom(\K^{-},\K)$ of endofunctors of $\Set$.

Now suppose that $A$ is a $\K$-function algebra, and let $\eps_A \colon A \to \K^{\Hom(A,\K)}$
denote the natural map to the product in $\TAlg_\K$, given by 
$\eps_A(f) = (\psi(f))_{\psi \in \Hom(A,\K)}$. This is an isomorphism according to
Lemma~\ref{lem:function algebra}(c).
Let $g \colon A \to B$ be a morphism in $\Func(\K)$. Denote $g^* = \Hom(g,\K) \colon
\Hom(B,\K) \to \Hom(A,\K)$, given by precomposition with $g$. 
Given $f \in A$, we now alternately view $\eps_A(f) \in \K^{\Hom(A,\K)}$ as the function
$\eps_A(f) \colon \Hom(A,K) \to \K$ given by $\psi \mapsto \psi(f)$. So fixing $\psi \in \Hom(B,\K)$,
we have 
\[
(\K^{g^*} \circ \eps_A(f)) (\psi) = (\eps_A(f) \circ g^*)(\psi) = \eps_A(f)(\psi \circ g) = \psi(g(f))
= \eps_B(g(f))(\psi).
\]
It follows that the diagram
\[
\xymatrix{
A \ar[rr]^-{\eps_A} \ar[d]_{g} & & \K^{\Hom(A,\K)} \ar[d]^{\K^{\Hom(g,\K)}} \\
B \ar[rr]^-{\eps_B} & & \K^{\Hom(B,\K)}
}
\]
commutes, making the $\eps_A$ into the components of a natural isomorphism $\eps \colon 
\mathbf{1}_{\Func(K)} \to \K^{\Hom(-,\K)}$ of endofunctors of $\Func(K)$. This establishes
the desired contravariant equivalence between $\Set$ and $\Func(K)$.
\end{proof}

For us, the key application of the previous theorem is to determine the structure of an arbitrary
closed subalgebra of a $\K$-function algebra.

\begin{corollary}\label{cor:subalg}
Let $\K$ be a field, $I$ a set, and $A \subseteq \K^I$ a closed $\K$-subalgebra. Then $A\cong \K^J$,
as topological $\K$-algebras, for some set $J$ with $|J| \leq |I|$.
\end{corollary}

\begin{proof}
Lemmas~\ref{lem:function algebra} and~\ref{lem:closed function subalg} yield that 
$A \cong \K^{\Spec_0(A)}$ is a function algebra, so that the inclusion 
map $A \hookrightarrow \K^I$ is a monomorphism in $\Func(\K)$. Thus applying $\Spec_0$ yields
an epimorphism in $\Set$ (i.e., a surjection) $I \cong \Spec_0(\K^I) \twoheadrightarrow \Spec_0(A)$. 
This implies that $|\Spec_0(A)| \leq |I|$. The claim follows by setting $J = \Spec_0(A)$.
\end{proof}

Even more precisely than the above, Theorem~\ref{thm:duality} allows one to characterize the closed
subalgebras of $\K^I$ (and their containments) in terms of equivalence relations $\sim$ on $I$ (and
their refinements), as closed subalgebras of $\K^I$ are dual to the surjections 
$I \twoheadrightarrow I/\sim$. We do not include further details as we will not make use of this
observation.

We also note that the above statement can be easily translated into a coalgebra statement and 
proved this way: a quotient of the coalgebra $C=\K^{(I)}$ has a basis of grouplike elements 
$(g_i)_{i\in I}$ and thus is isomorphic to $\K^{(J)}$ for some set of cardinality $|J| \leq |I|$. Indeed, 
if $p:C\rightarrow D$ is such a quotient, then each grouplike $g_i\in C$ produces a grouplike
$p(g_i)\in D$, and $D$ is spanned by these grouplike elements. Extracting a basis, we see 
that $D\cong \K^{(J)}$ as coalgebras. 

In general, the conclusion of the previous result fails for subalgebras of $\K^I$ that are not closed.

\begin{example}
Let $\K$ be an infinite field, and let $(\lambda_i)_{i\in I}$ be a tuple of elements in $\K$ for which 
the set $\{\lambda_i \mid i\in I\}$ is infinite. Then the element $\theta=(\lambda_i)_{i\in I}\in \K^I$
satisfies no polynomial in $\K[x]$, since $\{\theta^n : n \geq 0\}$ is linearly independent over $\K$,
and therefore the subalgebra $\K[\theta]\subseteq \K^I$ is isomorphic to $\K[x]$. In particular,
$\K[\theta]$ is not isomorphic to $\K^\Omega$ for any cardinal $\Omega$, since
$\K[\theta]$ is an integral domain and $\K^\Omega$ is not (for $\Omega > 1$).
\end{example}

\subsection{Characterizations of diagonalizable subalgebras}

We are now ready for our main result about diagonalization of algebras of operators.

\begin{theorem}\label{thm:diag}
Let $\K$ be a field, $V$ a $\K$-vector space, and $A \subseteq \End(V)$ a closed subalgebra. Then the following are equivalent:
\begin{enumerate}[label=(\alph*)]
\item $A$ is diagonalizable;
\item $A \cong \K^\Omega$ as topological $\K$-algebras;
\item $A$ is the closed subalgebra of $\End(V)$ generated by a summable set of orthogonal idempotents $\{E_i \mid i \in I\}$ with $\sum E_i = 1$;
\end{enumerate}
and the cardinal $\Omega$ in~(b) above must satisfy $\Omega \leq \dim(V)$.

Furthermore, if $\K$ is algebraically closed, then the above are also equivalent to:
\begin{enumerate}[label=(\alph*), resume]
\item $A$ is $\K$-pseudocompact and Jacobson semisimple.
\item $A$ is $\K$-pseudocompact and topologically reduced.
\end{enumerate}
\end{theorem}

\begin{proof}
Suppose that (a) holds. Let $\B$ be a basis of $V$ such that $A$ is 
$\B$-diagonalizable, so that  $A$ is a closed subalgebra of the algebra $C$ of all $\B$-diagonalizable
operators. Now $C \cong \K^\B$ as topological algebras according to  Proposition~\ref{prop:diagonal structure}. 
It follows from Corollary~\ref{cor:subalg} that $A \cong \K^\Omega$ as topological $\K$-algebras for 
some cardinal $\Omega \leq |\B| = \dim(V)$, and~(b) is established. 

Suppose that (b) holds. Since $\K^\Omega$ is the closure of the subalgebra generated by a summable
set of orthogonal idempotents, summing to $1$ (namely, the characteristic functions on the singletons
of $\Omega$), the same is true of $A$, and hence (c) holds.

Suppose that~(c) holds. Let $A_0 = \Span\{E_i \mid i \in I\} \subseteq \End(V)$, which is a 
commutative subalgebra. 
Then $A = \overline{A_0}$ is commutative by Lemma~\ref{lem:commutativeclosure}(1).
Furthermore, $V=\bigoplus E_i(V)$, by Corollary~\ref{cor:sumconditions}. Note that
each $V_i = E_i(V)$ is a simultaneous eigenspace for $A_0$; we claim that these are also
eigenspaces for $A$. Indeed, given any $i \in I$, let $T \in A$ and $v \in V_i$. By the density of $A_0$
in $A$, there exists $T_0 \in A_0$ that belongs to the open neighborhood $\{S \in A \mid T(v) = S(v)\}$
of $T$ in $A$. But since $V_i$ is an eigenspace of $T_0$, this means that $T(v) = T_0(v) 
= \lambda v$ for some $\lambda \in \K$. This confirms that each $V_i$ is a simultaneous 
eigenspace for $A$, from which~(a) follows.

Assuming that $\K$ is algebraically closed, the equivalence of~(b), (d), and~(e) follows from
Lemma~\ref{lem:function algebra}.
\end{proof}

In Corollary~\ref{cor:diagsubalg} below, we apply the above theorem to characterize
diagonalizable subalgebras of $\End(V)$ that are not necessarily closed, still in terms of
the restriction of the finite topology. We make use of the following preparatory fact.

\begin{lemma}\label{lem:dense quotient}
Let $B$ be a topological ring with a dense subring $A \subseteq B$. For any open ideal
$I$ of $B$, the canonical map $A/(A \cap I) \to B/I$ is a (topological) isomorphism.
\end{lemma}

\begin{proof}
The canonical map $A/(A \cap I) \to B/I$ given by $a + (A \cap I) \mapsto a + I$ for $a \in A$ is
certainly injective. As Lemma~\ref{lem:open} implies that both factor rings have the discrete topology
(because the ideals are respectively open in $A$ and $B$), it suffices to prove that this map is surjective. 
To this end, fix a coset $b+I \in B/I$. Note that this coset is open as it is a translate of the open set 
$I$, so there exists an element $a \in A \cap (b+I)$ by the density of $A$. 
But then $a + (A \cap I) \mapsto a+I = b+I$ as desired.
\end{proof}

\begin{corollary}\label{cor:diagsubalg}
Let $V$ be a vector space over a field $\K$, 
and let $A$ be a subalgebra of $\End(V)$, considered as a topological algebra with
the topology inherited from the finite topology.
Then the following are equivalent:
\begin{enumerate}[label=(\alph*)] 
\item $A$ is diagonalizable;
\item $\overline{A}\cong \K^\Omega$ as topological algebras for some cardinal $\Omega$, which
necessarily satisfies $\Omega \leq \dim(V)$;
\item For every open ideal $I$ of $A$, there is an integer $n \geq 1$ such that $A/I \cong \K^n$
as $\K$-algebras.
\end{enumerate}
\end{corollary}
\begin{proof}
First we show (a)$\Leftrightarrow$(b).
Assuming~(a), fix a basis $\B$ such that $A$ is $\B$-diagonalizable, and let $C$ denote the set of $\B$-diagonalizable
operators on $V$. This is a maximal commutative subalgebra of $\End(V)$ by Proposition~\ref{prop:diagonal structure}, 
hence closed thanks to Lemma~\ref{lem:commutativeclosure}(3). Thus $\Abar \subseteq C$ is also diagonalizable,
and~(b) (along with the bound $\Omega \leq \dim(V)$) follows from Theorem~\ref{thm:diag}. 
Conversely, if~(b) holds then $\Abar$ is diagonalizable by Theorem~\ref{thm:diag}. As 
$A \subseteq \overline{A}$, we conclude that~(a) holds.

Next we show (b)$\Rightarrow$(c). Assuming~(b), let $I$ be an open ideal of $A$. 
Because the topology of $A$ is induced from that of $\overline{A}$, there is an open ideal $J$ of
$\Abar$ such that $J \cap A \subseteq I$ and $\Abar/J \cong \K^m$ for some integer $m \geq 1$
(thanks to the structure of $\K^\Omega$ as a topological algebra). Using 
Lemma~\ref{lem:dense quotient}, this means that we have a surjection $\K^m \cong \Abar/J \cong
A/(J \cap A) \twoheadrightarrow A/I$, from which it follows that $A/I \cong \K^n$
for some integer $n \leq m$. This verifies~(c). 

Conversely, assume~(c) holds. Note that $\Abar$ is pro-discrete by
Lemma~\ref{lem:closed prodiscrete}. Let $J \subseteq \Abar$ be an open ideal. Then 
$\Abar/J \cong A/(J \cap A) \cong \K^n$ as $\K$-algebras thanks to Lemma~\ref{lem:dense quotient}
and the hypothesis. It follows that $\Abar$ is $\K$-pseudocompact, that every open maximal ideal 
of $\Abar$ has $\K$-codimension equal to~1 (as $\K^n$ is a field if and only if $n = 1$), and 
that every open ideal of $\Abar$ is an intersection of open maximal ideals (as the intersection of the
maximal ideals in the discrete algebra $\Abar/J \cong \K^n$ is zero). Because $\Abar$ is pro-discrete, 
the intersection of all of its open ideals is zero; as each of these open ideals is an intersection of open
maximal ideals, we deduce that $J_0(\Abar) = 0$. 
Now Lemma~\ref{lem:function algebra} implies that~(b) holds.
\end{proof}

The conditions on diagonalizability of a commutative subalgebra may also be translated into a condition for
an operator $T \in \End(V)$ to be diagonalizable. Condition~(b) below is a topological substitute for the
characterization in the finite-dimensional case that the minimal polynomial of an operator splits.

\begin{proposition}\label{prop:operator}
Let $V$ be a vector space over a field $\K$, and let $T \in \End(V)$. 
Let $\sigma = \sigma(T)$ denote the spectrum of eigenvalues of $T$. Then the following are equivalent:
\begin{enumerate}[label=(\alph*)]
\item $T$ is diagonalizable;
\item The net of finite products of the form $(T-\lambda_1) \cdots (T - \lambda_n)$ for distinct 
$\lambda_i \in \sigma$ (indexed by the finite subsets 
$\{\lambda_1, \dots, \lambda_n\} \subseteq \sigma$) converges to zero;
\item For every finite-dimensional subspace $W \subseteq V$, there are distinct
$\lambda_1, \dots, \lambda_n \in \sigma$ such that the restriction of $(T - \lambda_1) \dots 
(T - \lambda_n)$ to $W$ is zero;
\item The closed $\K$-subalgebra $\overline{\K[T]} \subseteq \End(V)$ generated by $T$ is isomorphic 
to $\K^\Omega$ as a topological $\K$-algebra for some cardinal $\Omega \leq \dim(V)$.
\end{enumerate}
\end{proposition}
\begin{proof}
Note that $T$ is diagonalizable if and only if the subalgebra $\K[T] \subseteq \End(V)$ is diagonalizable.
Thus (a)$\Leftrightarrow$(d) follows from Corollary \ref{cor:diagsubalg}. Note also that 
(b)$\Leftrightarrow$(c) because~(c) is simply a reformulation of~(b) in terms of the finite topology.

For (a)$\Rightarrow$(b), suppose that $T$ is $\B$-diagonalizable for some basis $\B$ of $V$. To prove~(b),
note that the set of open neighborhoods of $\End(V)$ of the form
\[
U = \{S \in \End(V) \mid S(b_1) =  \dots = S(b_n) = 0\}
\]
for some finite subset $\{b_1,\dots,b_n\}\subseteq \B$ forms a neighborhood basis of zero (since any 
finite-dimensional subspace of $V$ is contained in the span of a sufficiently large but finite subset of 
$\B$).
Given such $b_1,\dots,b_n \in \B$, let $\alpha_1,\dots,\alpha_n \in \sigma$ denote the 
(possibly repeated) eigenvalues of $T$ associated to the $b_i$. Then for any finite subset 
$X \subseteq \sigma$ containing $\{\alpha_1, \dots, \alpha_n\}$, set 
$S = \prod_{\lambda \in X} (T-\lambda)$. 
Because the factors $T-\lambda I$ commute with one another and each $\alpha_i \in X$, we have 
$S(b_i) = 0$ for $i = 1, \dots, n$ so that $S \in U$. Thus~(b) is satisfied.

Now suppose that (b) holds; we verify~(d). Let $I \subseteq \K[T]$ be an open ideal. By hypothesis 
there exist  distinct $\lambda_1, \dots, \lambda_n \in \sigma$ such that, for the polynomial
$p(x) = (x-\lambda_1) \cdots (x - \lambda_n) \in \K[x]$, we have $p(T) \in I$. An easy application of
the Chinese Remainder Theorem implies that $\K[x]/(p(x)) \cong \K^n$ as $\K$-algebras.
Thus there is a surjection $\K^n \cong \K[x]/(p(x)) \twoheadrightarrow \K[T]/(p(T)) \twoheadrightarrow
\K[T]/I$ of  $\K$-algebras, from which it follows that $\K[T]/I \cong \K^m$ for some $m \leq n$. 
Then~(d)  holds by Corollary~\ref{cor:diagsubalg}.
\end{proof}

\subsection{Simultaneously diagonalizable operators}

It is well known that if two diagonalizable operators on a finite-dimensional vector space $V$
commute, then they are simultaneously diagonalizable. An immediate corollary (taking into account
the finite-dimensionality of $\End(V)$) is that an arbitrary commuting set of diagonalizable operators
on $V$ is simultaneously diagonalizable.

A carefully formulated analogue of this statement passes to the infinite-dimensional case, 
but a counterexample shows that the statement does not fully generalize in 
the strongest sense. We begin with the positive results. The following may be well-known in
other contexts, as it can also be proved with an adaptation of the classical argument that two 
commuting diagonalizable operators on a finite-dimensional vector space can be simultaneously 
diagonalized. We provide an alternative argument via summability of idempotents.

\begin{theorem}
Let $C, D \subseteq \End(V)$ be diagonalizable subalgebras that centralize one another. Then
$C$ and $D$ are simultaneously diagonalizable, in the sense that the subalgebra
$\K[C \cup D] \subseteq \End(V)$ generated by both sets is diagonalizable.
\end{theorem}

\begin{proof}
Because a subalgebra of $\End(V)$ is diagonalizable if and only if its closure is diagonalizable,
and because $\K[\overline{C} \cup \overline{D}] \subseteq \overline{\K[C \cup D]}$, we may assume
without loss of generality that $C$ and $D$ are closed. In this case, by Theorem~\ref{thm:diag} both
$C$ and $D$ are respectively generated by orthogonal sets of idempotents
$\{E_i \mid i \in I\}$ and $\{F_j \mid j \in J\}$ such that $\sum E_i = 1 =  \sum F_j$.

Consider the set of pairwise products $\{E_i F_j \mid (i,j) \in I \times J\}$. By hypothesis, the $E_i$ 
and  $F_j$ pairwise commute, so that each $E_i F_j$ is again idempotent. For  $(i,j) \neq (m,n)$ 
we have that $E_i F_j$ is orthogonal to $E_m F_n$. Furthermore, from Lemma~\ref{lem:sumofproducts}
we find that
\begin{align*}
\sum\nolimits_j E_i F_j &= E_i \quad \mbox{for each } i \in I, \\
\sum\nolimits_i E_i F_j &= F_j \quad \mbox{for each } j \in J, \mbox{ and} \\
\sum\nolimits_{i,j} E_i F_j &= 1.
\end{align*}
So $\{E_i F_j\}$ is an orthogonal set of idempotents whose sum is~1, and it follows from
Theorem~\ref{thm:diag} that the closed subalgebra $A \subseteq \End(V)$ generated thereby 
is diagonalizable. 
But also each $E_i = \sum_j E_i F_j \in A$ and each $F_j = \sum_i E_i F_j \in A$. So 
$C, D \subseteq A$ and it follows that $\K[C \cup D] \subseteq A$ is diagonalizable.
\end{proof}

Specializing to the case when one or both of the subalgebras is generated by a single
operator, we immediately have the following.

\begin{corollary}
Any finite set of commuting diagonalizable operators in $\End(V)$ is simultaneously diagonalizable.
If $C \subseteq \End(V)$ and $T \in C'$ are diagonalizable, then the subalgebra $C[T] \subseteq \End(V)$
generated by $C$ and $T$ is diagonalizable.
\end{corollary}

Of course, an infinite set of commuting diagonalizable operators need not be simultaneously 
diagonalizable. Such a situation is provided by Example \ref{ex:not summable}. That maximal
commutative subalgebra $A$ contains orthogonal idempotents $\{E_n \mid n = 0, 1, 2, \dots\}$;
these are diagonalizable and commute. If they were simultaneously diagonalizable, then the 
algebra $A$ would be diagonalizable as it is generated by the $E_n$. Hence, we would have
$\Abar \cong \K^\Omega$ for $\Omega$ an infinite cardinal. But $A=\Abar$ by
Lemma~\ref{lem:commutativeclosure}(3), so $\Abar$ is countable-dimensional. This yields a 
contradiction, because $\dim(\K^\Omega)$ is uncountable.

Next we will present in Example~\ref{ex:not simultaneously diagonalizable} another construction 
of an infinite set of commuting diagonalizable operators need not be simultaneously diagonalizable, 
of a somewhat more subtle nature. 
For any integer $n \geq 0$, we let $\two^n = \{0,1\}^n$ denote the set of strings of length $n$ 
in the  alphabet $\{0,1\}$. For instance, we have $\two^0 = \{\varnothing\}$, $\two^1 = \{0,1\}$, 
and $\two^2 = \{00,01,10,11\}$. 
Given $i \in \two^n$ and $j \in \two^p$, we let $ij \in \two^{n+p}$ denote the concatenated 
string that consists of $i$ followed by $j$. 

\begin{lemma}
\label{lem:decomposition}
Let $V$ be a vector space over a field $\K$ with a countable basis $\{v_1, v_2, v_3, \dots\}$.
For every integer $n \geq 0$, there exist subspaces $V_i \subseteq V$ for all $i \in \two^n$ 
with $V_\varnothing = V$ satisfying the following conditions:
\begin{enumerate}[label=(\alph*)]
\item $\Span(v_1,\dots,v_n) \cap V_i = 0$ for every $i \in \two^n$;
\item $V_i = V_{i0} \oplus V_{i1}$ for every $i \in \two^n$;
\item $\dim(V_i) = \aleph_0$ for every $i \in \two^n$;
\item There is $w\in V$ such that $w\not\in \bigoplus_{j\in\two^n-\{i\}}V_j$, for all $i\in \two^n$.
\end{enumerate}
Consequently, $V = \bigoplus_{i \in \two^n} V_i$ for any $n \geq 0$, and for any sequence of
bits $b_1, b_2, \dots \in \{0,1\}$, we have $\bigcap_{n=1}^\infty V_{b_1 b_2 \dots b_n} = \{0\}$.
\end{lemma}

\begin{proof}
We proceed by induction. The case $n = 0$ is covered by simply setting $V_\varnothing = V$, and 
choosing $w=v_1$.

For the inductive step, assume that we have constructed $V_i$ for all strings $i$ of length less
than $n$. Let $i \in \two^{n-1}$; we will define $V_{i0}$ and $V_{i1}$. Since $\Span(v_1,\dots,v_{n-1})\cap V_i=0$, 
it follows that $\Span(v_1,\dots,v_n)\cap V_i$ can be at most $1$-dimensional. Let also 
$w=\sum_{i\in\two^{n-1}}w_i$ be the decomposition of $w$ with respect to $V=\bigoplus_{i\in\two^{n-1}}V_i$; 
by the inductive hypothesis, all $w_i$ are nonzero. In order to construct $V_{i0}$ and $V_{i1}$ in such a way
as to satisfy condition~(d), we distinguish two cases: 
\begin{enumerate}
\item $\Span(v_1, \dots, v_n) \cap V_i \subseteq \Span(w_i)$. 
We write $w_i=a+b$ inside $V_i$, with $a,b$ independent, 
and we choose $V_{i0}$ and $V_{i1}$ to be (countably) infinite-dimensional subspaces of $V_i$ which 
split $V_i$ such that $a\in V_{i0}$ and $b\in V_{i1}$ (by completing $\{a,b\}$ to a basis and splitting 
the basis appropriately). 
\item $\Span(v_1,\dots,v_n)\cap V_i=\Span(g)$ and $g,w_i$ are linearly independent. In this case, let $a,b\in V_i$ be such that $\{g,w_i,a,b\}$ are linearly independent. Then the set $\{g-a,w_i-b,a,b\}$ is also linearly independent, and we can choose $V_{i0}, V_{i1}$ be infinite-dimensional with $V=V_{i0}\oplus V_{i1}$ and $g-a,w_i-b\in V_{i0}$ and $a,b\in V_{i1}$. This shows that $w_i=(w_i-b)+(b)$ (and also $g=(g-a)+(a)$) has non-zero components in $V_{i0}$ and $V_{i1}$, and $\Span(v_1,\dots,v_n)\cap V_{i0}=\Span(v_1,\dots,v_n)\cap V_i \cap V_{i0}=\Span(g)\cap V_{i0}=0$ and similarly $\Span(v_1,\dots,v_n) \cap V_{i1} = \{0\}$.
\end{enumerate}
This completes the proof by induction.

It remains to verify the final two claims of the statement.
For any $n \geq 1$, the direct sum decomposition $V = \bigoplus_{i \in \two^n} V_i$ 
follows from $V_\varnothing = V$ and condition~(b). Finally, given $b_1, b_2, \dots \in \{0,1\}$, 
using condition~(a) we get that $V_{b_1 \cdots b_n} \cap \Span(v_1,\dots,v_n) = \{0\}$.
Since $V = \bigcup_{n=1}^\infty \Span(v_1,\dots,v_n)$, we conclude that
$\bigcap_{n=1}^\infty V_{b_1 b_2 \dots b_n} = \{0\}$.
\end{proof}

Let $\two^* = \bigsqcup_{n=0}^\infty \two^n$ denote the set of all words in the alphabet 
$\two = \{0,1\}$.

A version of the lemma above holds for a vector space $V$ of arbitrarily large infinite dimension, 
if we delete condition (d). This can be shown by decomposing $V = \bigoplus_{j \in \dim(V)} V_j$
as a direct sum of vector spaces of dimension $\aleph_0$, constructing subspaces 
$W_{j,i} \subseteq V_j$ for each $i \in 2^*$ that satisfy (a)--(c), and then setting 
$V_i = \bigoplus_j W_{j,i}$.

\begin{example}\label{ex:not simultaneously diagonalizable}
Let $V$ be a $\K$-vector space of countably infinite dimension. Fix subspaces $V_i \subseteq V$
for all $i \in \two^*$ as in Lemma~\ref{lem:decomposition}. 
Fixing $n \geq 0$ and $i \in \two^n$, let $E_i \in \End(V)$ be the idempotent whose range is $V_i$ 
and whose kernel is $W_i = \bigoplus_{j\in \two^n \setminus \{i\}} V_j$. Then the 
$\{E_i \mid i \in \two^n\}$ form orthogonal sets of idempotents such that $1 = \sum_{i \in \two^n} E_i$.
Set $A_n = \bigoplus_{i \in \two^n} \K E_i \subseteq \End(V)$, the commutative subalgebra generated 
by the idempotents indexed by $\two^n$. 
Condition~(b) of Lemma~\ref{lem:decomposition} furthermore ensures that each $E_i = E_{i0} + E_{i1}$.
Thus we have $A_n \subseteq A_{n+1}$ for all $n$. Then $A = \bigcup A_n$ is a commutative
subalgebra of $\End(V)$, generated by the infinite set $\{E_i \mid i \in \two^n, n \geq 0\}$ 
of commuting idempotents. Being idempotent operators, these generators are diagonalizable.

We claim that \emph{the algebra $A$ is not diagonalizable,} and so the idempotents $E_i$ for 
$i\in \bigcup_{n\geq 0}\two^n$ cannot be simultaneously diagonalized. In fact, they have no common 
eigenvector (i.e. $V$ has no 1-dimensional $A$-submodule). Indeed,  assume for contradiction
that $v \in V \setminus \{0\}$ is an eigenvector for all of the idempotents $E_i$. The eigenvalue 
$\lambda_i$ of $E_i$ corresponding to the vector $v$ is either $0$ or $1$. From 
$1_V = \sum_{i \in \two^n} E_i$ we see that, for $i \in \two^n$, exactly one of the $E_i$ has 
eigenvalue~$1$ and the rest have eigenvalue~$0$; this means that $v \in V_{i_n}$
for exactly one $i_n \in \two^n$. Note that each $i_n = i_{n-1} b_n$ for some $b_n \in \{0,1\}$.
This means that there is a sequence $b_1, b_2, b_3, \dots \in \{0,1\}$ such that 
$v \in \bigcap_{n = 0}^\infty V_{b_1 \cdots b_n}$, which contradicts the fact that the latter
intersection is zero by condition~(b) of Lemma~\ref{lem:decomposition}.

Finally, we claim that \emph{$A$ is a discrete subalgebra of $\End(V)$.} 
Indeed, let $w$ denote the vector provided by Lemma~\ref{lem:decomposition}(d). 
Let $E$ be an idempotent such that $w\in \ker(E)$ and $\ker(E)$ is finite-dimensional. Then the left ideal 
$I=\End(V)E$ is open. We show that $I\cap A=0$, which will imply that $0$ is open, so $A$ is discrete. 
If $I\cap A\neq 0$ then $I\cap A_n\neq 0$ for some $n$. Let $T=\sum_{i \in \two^n} \lambda_i E_i 
\in I \cap A_n$. If $F=\{i\in\two^n \mid \lambda_i\neq 0\}$, then $F\neq \two^n$ since otherwise 
$T$ would be invertible. 
Also $\ker(\sum_{i\in F}\lambda_i E_i)=\bigoplus_{i\not\in F}V_i$, and since $E(w)=0$, it follows that 
$w \in \bigoplus_{i\not\in F}V_i$. But by construction, $w$ has non-zero components in all terms 
$\bigoplus_{i\in \two^n} V_i$. This contradicts the fact that $F\neq \two^n$, completing the argument. 
\end{example}

We note a few more properties of the algebra $A$. We remark that as $A$ is the union of the $A_n$ 
which are von Neumann regular, the same property holds for $A$. In fact, algebraically each
$A_n\cong \K^{\two^n}$, and the inclusion $A_n\subseteq A_{n+1}$ corresponds to the diagonal
embedding of $\K^{\two^n}$ into 
$\K^{\two^{n+1}}=\K^{\two^n0\,\sqcup \,\two^n1}\cong \K^{\two^n}\times \K^{\two^n}$ (i.e., with 
each  $\K\rightarrow \K\times\K$, $1\mapsto (1,1)$). 
Now because $A$ is commutative and von Neumann regular, it is reduced. Being both reduced and
discrete, we see that $A$ is also topologically reduced.

We leave open the interesting question of what is the closure of $A$ within $\End(V)$, especially
to what extent this depends on the particular choice of the subspaces $V_i$ for $i \in \two^*$.
It would also be interesting to understand to what extent the closure would change if one omitted
condition~(d) from Lemma~\ref{lem:decomposition}, which forced $A$ to be discrete.

\subsection{Closure of the set of diagonalizable operators}

Finally, we examine the closure of the set $\D(V)$ of diagonalizable operators within $\End(V)$. 
If $T\in\End(V)$, let us denote by $H(T)$ the set of all $v\in V$ for which $p(T)v=0$ for some nonzero 
polynomial $p(x) \in \K[x]$. If we consider $V$ as a $\K[x]$-module with $x$ acting as $T$, this 
means that $H(T)$ is the torsion part of $V$, or equivalently, the sum of all of its finite-dimensional 
$\K[x]$-submodules. We will show that the $T$-invariant subspace $H(T)$ plays an important role 
in the characterization of $\overline{\D(V)}$ in the case of vector spaces over infinite fields. 

We need the following easy remark: if $W$ is a finite-dimensional vector space over 
an infinite field $\K$ with basis $v_0,v_1,\dots,v_n$, then there is $w\in W$ such that the linear map 
$L\in\End(W)$ defined by $L(v_i)=v_{i+1}$ for $i<n$ and $L(v_n)=w$ is diagonalizable; we show that 
in fact we can choose $w=\sum_{i=0}^na_iv_i$ with $a_0,a_1,\dots,a_n\in\K$. Indeed, for this it is 
enough to show that the characteritic polynomial $f_L$ of $L$ has all simple roots in $\K$. 
But $f_L$ has exactly the coefficients $-a_i$, and we can choose the $a_i$ as the coefficients of the
polynomial $\prod_{i=0}^n(x-\lambda_i)$, where $\lambda_i$ are pairwise distinct elements in $\K$.  
We have the following.

\begin{theorem}
Let $\K$ be an infinite field and $T\in\End(V)$. Then $T \in \overline{\D(V)}$ if and only if $T$ is 
diagonalizable on $H(T)$.
\end{theorem} 
\begin{proof}
If $T \in \overline{\D(V)}$, then for every finite-dimensional $T$-invariant subspace $W$ of 
$H(T)$, we can find $D\in\D(V)$ such that $D=T$ on $W$. Since $W$ is $T$-invariant, it is 
$D$-invariant, and $D$ is diagonalizable on $W$ since it is on $V$ ($W$ is a submodule of the 
semisimple $D$-module $V$). Hence $T$ is diagonalizable on $W$, and $W$ is a (possibly trivial)
sum of $1$-dimensional $T$-invariant subspaces. As $H(T)$ is the sum of all its finite-dimensional 
subspaces, it follows that it is 
a sum of $1$-dimensional $T$-invariant subspaces. This is to say that $H(T)$ is spanned by a set
of $T$-eigenvectors. This spanning set contains a basis of $H(T)$. Thus $H(T)$ has a 
basis consisting of $T$-eigenvectors, and $T$ is diagonalizable on $H(T)$. 

Conversely, assume $T$ is diagonalizable on $H(T)$. Let $W\subseteq V$ be finite-dimensional, 
and write $W=(H(T)\cap W)\oplus U$. Consider $V$ as a $\K[x]$-module under the action of $T$. 
Then $M=\K[x]W$ is finitely generated, and its torsion part is $H(T) \cap M = H(T) \cap W$. 
Hence, we may find a free $\K[x]$-module $N$ ($T$-invariant subspace) such that 
$M=(M\cap H(T)) \oplus N$, with $N\cong \K[x]^t$. Let $w_1,\dots,w_t$ be a $\K[x]$-basis 
of $N$. Then a $\K$-basis for $N$ is given by $\{T^i(w_k) \mid i\geq 0,\ 1\leq k \leq t\}$, 
and because $W$ is finite-dimensional there exists $n$ such that $W\subseteq (W\cap H(T)) \oplus 
\Span\{T^i(w_k) \mid i<n, 1 \leq k \leq t\}$. Let $D\in \End(V)$ be the linear map defined as follows:
\begin{itemize}
\item $D$ equals $T$ on $H(T) \cap W$;
\item $D$ equals T on the finite set $\{T^i(w_k) \mid i<n, 1\leq k\leq t\}$ and 
  $D(T^n(w_k))=u_k$ where $u_k$ is chosen as above such that 
  $u_k \in W_k = {\rm Span}\{T^i(w_k) \mid i\leq n\}$ and $D$ is diagonalizable 
  on $W_k$;
\item $D$ equals $0$ on a complement $L$ of $(H(T) \cap W) \oplus \left(\bigoplus W_k\right)$.
\end{itemize}
By construction we have $T$ equal to $D$ on $W \subseteq (H(T) \cap W) \oplus 
\left(\bigoplus W_k\right)$. Since $D$ is diagonalizable on each of the invariant subspaces 
$H(T) \cap W$, $W_k$, and $L$, it follows that $D$ is diagonalizable on $V = (H(T) \cap W) 
\oplus \left(\bigoplus W_k\right) \oplus L$. This shows that every open neighborhood of 
$T$ contains some diagonalizable $D$ (since $D=T$ on $W$), and the proof is finished. 
\end{proof}

The above shows that in fact the closure of the set of diagonalizable operators coincides 
with the closure of the set of diagonalizable operators of \emph{finite rank} on $V$.

We give an example below to show that the above characterization of the closure of
$\D(V)$ fails if the field $\K$ is finite. Before doing so, we will in fact show that 
$\D(V)$ is closed if the field of scalars is finite. Given a prime power $q$, we let
$\F_q$ denote the field with $q$ elements.

\begin{proposition}
Let $V$ be a vector space over a finite field $\F_q$. An operator $T \in \End(V)$ is diagonalizable if 
and only if it satisfies $T^q = T$. Consequently, the set $\D(V)$ of diagonalizable operators is closed 
in $\End(V)$.
\end{proposition}

\begin{proof}
Consider the polynomial $p(x) = x^q - x \in \F_q[x]$. As every element of $\F_q$ is a root of this 
polynomial, the same is certainly true for any diagonalizable operator on $V$. Thus if $T \in \End(V)$ 
is diagonalizable then $T^q = T$. Conversely, suppose that $T$ satisfies $T^q = T$, so that $p(T) = 0$ 
in the algebra $\End(V)$. It follows that $T$ has a minimal polynomial $\mu(x)$ (in the usual sense) 
that divides $p(x) = \prod_{\lambda \in \F_q} (x - \lambda)$. 
Thus $\mu(x)$ splits into distinct linear factors over $\F_q$, and it follows from Proposition~\ref{prop:operator} that $T$ is diagonalizable.

Finally, $\D(V)$ is closed because it is the zero set of $p(x)$ interpreted as a continuous function 
$p \colon \End(V) \to \End(V)$.
\end{proof}

\begin{example}
Let $V$ be a vector space over a finite field with basis $\{v_i \mid i = 1, 2, \dots\}$, and 
consider the right shift operator $S \in \End(V)$ with $S(v_i) = v_{i+1}$ for all $i$. It is clear 
from Example~\ref{ex:right shift} that $H(S) = 0$, so that $S$ is diagonalizable on $H(S)$. 
However, as $S$ is not diagonalizable (having no eigenvectors), we have 
$S \notin \D(V) = \overline{\D(V)}$ thanks to the proposition above.
\end{example}

\bibliographystyle{amsplain}
\bibliography{infdiagss-arxiv-v3}

\end{document}